\documentclass[10pt,reqno,final]{amsart}
\usepackage[notcite,notref]{showkeys}
\usepackage{url,hyperref,multirow}
\usepackage{color}
\usepackage{epsfig,subfigure,amssymb,amsmath,version,slashbox}
\usepackage{amssymb,version,graphicx,fancybox,mathrsfs,pifont,booktabs}

\catcode`\@=11 \theoremstyle{plain}
\@addtoreset{equation}{section}   
\renewcommand{\theequation}{\arabic{section}.\arabic{equation}}
\@addtoreset{figure}{section}
\renewcommand\thefigure{\thesection.\@arabic\c@figure}
\@addtoreset{table}{section}
\renewcommand\thetable{\thesection.\@arabic\c@table}

\newenvironment{theorem}{\begin{thm}} {\end{thm}}
\newtheorem{thm}{\bf Theorem}[section]
\newtheorem{cor}{\bf Corollary}[section]
\newtheorem{prop}{Proposition}[section]

\newtheorem{lmm}{\bf Lemma}[section]

\newenvironment{lemma}{\begin{lmm}}{\end{lmm}}
\theoremstyle{remark}
\newtheorem{rem}{Remark}[section]

\newcommand{\re}[1]{(\ref{#1})}

\def \ri {{\rm i}}

\def \af {\alpha}

\def \dnk {d_{n,k}^{\lambda}}

\begin{document}
\baselineskip 13pt
\bibliographystyle{plain}
\graphicspath{{./figs/}}

\title[Exponential Convergence of Gegenbauer Spectral Differentiation]
{On Exponential Convergence of Gegenbauer Interpolation and Spectral
Differentiation}
\author[Z. Xie, ~ L.   Wang~  $\&$~ X. Zhao~~] {Ziqing Xie${}^{1}$, \quad
  Li-Lian Wang${}^{2}$ \quad
and \quad    Xiaodan Zhao${}^{2}$}
\thanks{\noindent${}^{1}$ School of Mathematics and Computer Science, Guizhou Normal
University, Guiyang, Guizhou 550001, China; College of Mathematics and Computer Science, Hunan Normal University,
Changsha, Hunan 410081, China. The research of the author is partially supported by the
NSFC (10871066) and the Science and Technology Grant  of Guizhou Province  (LKS[2010]05). \\
\indent${}^{2}$ Division of Mathematical Sciences, School of Physical
and Mathematical Sciences,  Nanyang Technological University,
637371, Singapore. The research of the authors is partially
supported by Singapore AcRF Tier 1 Grant RG58/08,  Singapore MOE
Grant T207B2202 and Singapore NRF2007IDM-IDM002-010.
 }
\date{\today}
\begin{abstract}  This paper is devoted to a rigorous  analysis of exponential convergence of polynomial interpolation and spectral
differentiation based on the Gegenbauer-Gauss and Gegenbauer-Gauss-Lobatto points, when
the underlying function is analytic on and within an ellipse. Sharp error estimates in the
maximum norm  are derived.
\end{abstract}

\maketitle

\thispagestyle{empty}

\section{Introduction}
Perhaps the most significant advantage of the spectral method is its
high-order of accuracy. The typical convergence rate of the spectral
method  is $O(n^{-m})$ for every $m,$ provided that the underlying
function is sufficiently smooth
\cite{gottlieb1977numerical,MR1470226,Boyd01,Guo.W04,CHQZ06}. If the
function is suitably analytic, the expected rate is $O(q^n)$ with
$0<q<1.$ This  is the so-called {\em exponential convergence,} which
is  well accepted among the community.  There has been much
investigation on exponential decay of spectral expansions of
analytic functions. For instance, the justification for Fourier
and/or Chebyshev series can be found in
\cite{szeg75,davis1975interpolation,rivlin1990chebyshev,Boyd94,MR1937591,TrefSIAMRev08}.
In the seminal work   of Gottlieb and Shu
\cite{gottlieb1992gibbs,Got.S97},  on the resolution of the Gibbs
phenomenon (the interested readers are referred to Gustafsson
\cite{gustafsson2011work} for a review of this significant
contribution),  the exponential convergence, in the maximum norm
(termed as the so-called regularization error), of Gegenbauer
polynomial expansions was derived, when the index (denoted by
$\lambda$ below) grows linearly with the degree $n$. Boyd
\cite{boyd2005trouble} provided an insightful study of the
``diagonal limit" (i.e., $\lambda=\beta n$ for some constant
$\beta>0$)  convergence of the Gegenbauer reconstruction algorithm
in \cite{gottlieb1992gibbs}.  We remark that  under the assumption of
analyticity in \cite{gottlieb1992gibbs}, the exponential accuracy of
Gegenbauer expansions is actually   valid for fixed $\lambda$ (see
Appendix \ref{AppendixAdd} for the justification).

It is known that the heart of a  collocation/pseudospectral method
is the spectral differentiation process. That is, given a set of
Gauss-type points $\{x_j\}_{j=0}^n,$ e.g., on $[-1,1],$ the
derivative values $\{u'(x_j)\}$ can be  approximated by an exact
differentiation of the polynomial interpolant  $\{(I_nu)'(x_j)\}.$
Such a direct differentiation technique is also called a
differencing method in the literature (see, e.g.,
\cite{tadmor1986exponential,fornberg1998practical,ReddyWeideman2005}).
For the first time,  Tadmor \cite{tadmor1986exponential}  showed the
exponential accuracy of differencing analytic functions on
Chebyshev-Gauss points, where the main argument was based on
analyzing the (continuous) Chebyshev coefficients and the aliasing
error, and where the intimate relation between Fourier and Chebyshev
basis functions played an essential role in the analysis. Reddy and
Weideman \cite{ReddyWeideman2005} took a different approach and
improved the estimate in \cite{tadmor1986exponential} for Chebyshev
differencing of functions analytic on and within an ellipse
with foci $\pm 1.$
 As pointed out in \cite{ReddyWeideman2005}, although the exponential convergence  of
 spectral differentiation of analytic functions is  appreciated and mentioned in the literature
 (see, e.g., \cite{fornberg1998practical,Tref00,Boyd01}), the rigorous proofs (merely for
 the Fourier and Chebyshev methods) can only be found in  \cite{tadmor1986exponential,ReddyWeideman2005}.
Indeed, to the best of our knowledge, the theoretical justification
even for  the Legendre method is lacking.
It is worthwhile to point out that under the regularity  condition (M): $\|u^{(k)}\|_{L^\infty}\le cM^{k},$
the super-geometric convergence of  Legendre spectral
differentiation was proved by Zhang \cite{ZhangZM04}:
\begin{equation}\label{geometric}
\max_{0\le j\le n}|(u-I_nu)'(x_j)|\le
C\Big(\frac{eM}{2(n+1)}\Big)^{n+2},
\end{equation}
where $\{x_j\}_{j=0}^n$ are the Legendre-Gauss points. Similar
estimate was   nontrivially extended to the Chebyshev collocation
method in \cite{ZhangZM08}. The condition  (M) covers a large class
of functions, but it is
 even more  restrictive  than analyticity. On the other hand,  the regularity index $k$ could be infinite, while the dependence of  \eqref{geometric} on $k$ is not clear.

The main concern of this paper is to show  exponential convergence,
in the maximum norm, of Gegenbauer interpolation and spectral
differentiation on Gegenbauer-Gauss and Gegenbauer-Gauss-Lobatto
points, provided that the underlying function is analytic on and
within a sizable ellipse. The essential argument is based on the
classical Hermite's contour integral (see \eqref{Hermite} below),
and a delicate estimate of the asymptotics of the Gegenbauer
polynomial on the ellipse of interest. It is important  to remark
that the Chebyshev polynomial on the ellipse takes a very simple
explicit form (see, e.g., \cite{davis1975interpolation} or
\eqref{chebform} below), but the Gegenbauer polynomial has a
complicated expression.  Accordingly,  compared with  the Chebyshev
case in \cite{ReddyWeideman2005}, the analysis in this paper is much
more involved. The Chebyshev and Legendre methods are commonly used
in spectral approximations, but we have also witnessed renewed
applications of the Gegenbauer (or more general Jacobi) polynomial
based methods in, e.g.,   defeating Gibbs phenomenon (see, e.g., \cite{gottlieb1992gibbs,gelb2006determining}), $hp$-elements
(see, e.g., \cite{Dub.91,Bab.G01,Kar.S05,GuoBQ2011}), and numerical solutions of differential equations
(see, e.g., \cite{BernardDaugbook1999,GuoBy2001,Guo.W00,Doha91,Doha98,DohaBhrawy06,DohaBhrawy09,DohaBhrawyAbd09,WangGuo08,ShenTangWang2011})
and integral equations (see, e.g., \cite{Chen.T10}).
The results in this paper might be useful for a better understanding
of the methods and have implications in other applications.

The rest of the paper is organized as follows. As some
preliminaries, we briefly review  basic properties of Gegenbauer
polynomials, Gamma functions  and analytic functions in  Section 2. We study the asymptotics of the Gegenbauer polynomials in
Section 3, and present the main results on exponential convergence
of interpolation and spectral differentiation, together with some
numerical results and extensions in the last section.


\vspace*{-0.55cm}

\section{Preliminaries}
In this section, we collect  some relevant properties of Gegenbauer polynomials
and  assorted facts  to be used throughout the paper.

\subsection{Gegenbauer polynomials} The analysis heavily relies on  the normalization of \cite{szeg75},
so we define  the Gegenbauer polynomials\footnote{Historically, they are sometimes called ``ultraspherical polynomials" (see, e.g., the footnote on Page 80 of  \cite{szeg75} and Page 302 of \cite{Andrews99}).}
 by  the three-term recurrence:
\begin{equation}\label{PropG3}
\begin{split}
&nC_{n}^{\lambda}(x)=2\left(n+\lambda-1\right)xC_{n-1}^{\lambda}(x)-(n+2\lambda-2)C_{n-2}^{\lambda}(x),
\quad n\ge 2,\\
&C_{0}^{\lambda}(x)=1,\quad C_{1}^{\lambda}(x)=2\lambda x,\quad \lambda>- 1/2,\;\; x\in [-1,1].
\end{split}
\end{equation}
Notice that if $\lambda=0,$  $C_{n}^{\lambda}(x)$
vanishes identically for $n\ge 1.$ This corresponds to the
Chebyshev polynomial, and there
holds
\begin{equation}\label{chebdefn}
\lim_{\lambda\to 0} \lambda^{-1} C_n^\lambda (x)=\frac 2 n
T_n(x)=\frac 2 n\cos(n\ {\rm arccos}(x)),\quad n\ge 1.
\end{equation}
Hereafter, if not specified explicitly, we assume  $\lambda\not =0,$ and  refer
 to \cite{ReddyWeideman2005} for the analysis of the Chebyshev case.
 Notice that for $\lambda=1/2,$ $C^{\lambda}_n(x)=L_n(x),$ i.e., the usual Legendre polynomial of degree $n.$

The Gegenbauer polynomials  are orthogonal with respect to the weight function
$(1-x^2)^{\lambda-1/2},$ namely,
\begin{equation}\label{jacobi_orth}
\int_{-1}^1 C_n^{\lambda}(x) C_m^{\lambda}(x) (1-x^2)^{\lambda-1/2} dx
=h_n^{\lambda} \delta_{mn},
\end{equation}
where $\delta_{mn}$ is the Kronecker symbol, and
\begin{equation}\label{gammafd}
h_n^{\lambda}=\frac{2^{1-2\lambda}
\pi}{\Gamma^2(\lambda)}\frac{\Gamma(n+2\lambda)}{n!(n+\lambda)}.
\end{equation}
Moreover, we have
\begin{equation}\label{PropG1}
C_n^{\lambda}(-x)=(-1)^nC_n^{\lambda}(x),\quad
C_n^{\lambda}(1)=\frac{\Gamma{(n+2\lambda)}}{n!\Gamma{(2\lambda)}},
\end{equation}
and
\begin{equation}\label{PropG2}
\frac d {dx} C_n^{\lambda}(x)=2\lambda C_{n-1}^{\lambda+1}(x).
\end{equation}

By Formula (4.7.1) and Theorems 7.32.1 and 7.33.1 of Szeg$\ddot{o}$ \cite{szeg75},  we have
\begin{equation}\label{Gegenmaxim}
|C_n^{\lambda}(x)|\le C_n^\lambda(1) \;\;\;  {\rm if}\;  \lambda>0; \quad
|C_n^{\lambda}(x)|\le D_\lambda\, n^{\lambda-1} \;\;\;  {\rm if}\; -\frac 1 2 <\lambda<0 \;{\rm and}  \; n\gg 1,
\end{equation}
where $D_\lambda$ is a positive constant independent of $n.$ 
A  tight bound can be found in \cite{NevaiPEr1994} (also see
 \cite{krasikov2007upper}):
 \begin{equation}\label{sharpest}
 \max_{|x|\le 1}\Big\{(1-x^2)^{\lambda} \big(C_n^\lambda(x)\big)^2 \Big\}\le \frac{2e(2+\sqrt 2 \lambda)}{\pi} h_n^\lambda,\quad \lambda>0,\;  n\ge 0.
 \end{equation}


\subsection{Gamma and incomplete Gamma functions}
The following properties of the Gamma and incomplete Gamma  functions (cf.
\cite{watson66}) are found useful.
The Gamma
function satisfies
\begin{equation}\label{gammfunP1}
\Gamma(x)\Gamma(x+1/2)=2^{1-2x}\sqrt \pi\, \Gamma(2x),\quad \forall x\ge 0,
\end{equation}
and
\begin{equation}\label{gammfunP3}
\Gamma{(x)}\Gamma{(-x)}=-\frac{\pi}{x\sin(\pi x)},\quad \forall x>0.
\end{equation}
We would like to quote the  Stirling's formula (see, e.g.,
\cite{Got.S97}):
\begin{equation}\label{stirling}
 \sqrt{2\pi}x^{x+1/2}e^{-x}\leq\Gamma(x+1)\leq \sqrt{2\pi}x^{x+1/2}e^{-x}e^{\frac{1}{12x}},\quad
\forall x\geq 1.
\end{equation}
We also need to use  the incomplete Gamma function  defined by
\begin{equation}\label{incompleteGamma}
\Gamma(\alpha, x)=\int_x^\infty t^{\af-1} e^{-t} dt, \quad \af>0,\;\; x\ge 0,
\end{equation}
which satisfies (see P. 899 of \cite{Jeffrey:1086465})
\begin{equation}\label{incompleteGammaP1}
\Gamma(n+1, x)=n! e^{-x}\sum_{k=0}^n \frac{x^k}{k!},\quad n=0,1,\cdots.
\end{equation}

\subsection{Basics of analytic functions}  Suppose that
$u(x)$ is analytic on $[-1,1].$ Based on the notion of analytic
continuation, there always exists
 a simple connected region $R$ in the complex plane  containing $[-1,1]$ into which $f(x)$ can be continued analytically.  The analyticity  may be characterized by the growth of the derivatives of $u$. More precisely, let ${\mathcal C}$ be a
 simple positively oriented closed contour surrounding $[-1,1]$ and lying in $R.$  Then we have (see, e.g.,
 \cite{davis1975interpolation}):
 \begin{equation}\label{Anconnet}
\frac{|u^{(m)}(x)|}{m!}\le \frac{\max_{z\in {\mathcal C}}|u(z)|L({\mathcal C})}{2\pi \delta^{m+1}},\quad \forall x\in [-1,1],
 \end{equation}
 where $L({\mathcal C})$ is the length of ${\mathcal C},$ and $\delta$ is the distance from ${\mathcal C}$ to $[-1,1]$
 (can be viewed as the distance from  $[-1,1]$ to the nearest singularity of
 $u$ in the complex plane).   Mathematically,  an appropriate  contour to characterize the analyticity  is the so-called Bernstein ellipse:
\begin{equation}\label{Berellips}
{\mathcal E}_\rho:=\Big\{z\in {\mathbb C}~:~ z=\frac 1
2(w+w^{-1})\;\; \text{with}\;\; w=\rho e^{\ri \theta},\; \theta\in
[0,2\pi] \Big\},\quad \rho>1,
\end{equation}
where ${\mathbb C}$ is  the set of all complex numbers, and
$\ri=\sqrt{-1}$ is the complex unit. The ellipse ${\mathcal E}_{\rho}$ has
the foci at
 $\pm 1$ and the  major and minor semi-axes are, respectively
 \begin{equation}\label{semiaxis}
 a=\frac 1 2 \big(\rho+\rho^{-1}\big),\quad  b=\frac 1 2 (\rho-\rho^{-1}),
 \end{equation}
so the sum of two axes is $\rho.$ As illustrated in Figure \ref{Afigrr},   $\rho$ is the radius of the circle
 $w=\rho e^{\ri \theta} $ that is mapped to the ellipse ${\mathcal E}_\rho$
under the conformal mapping:  $z=\frac 1 2(w+w^{-1}).$
\begin{figure}[!th]
  \begin{center}
    \includegraphics[width=.4 \textwidth]{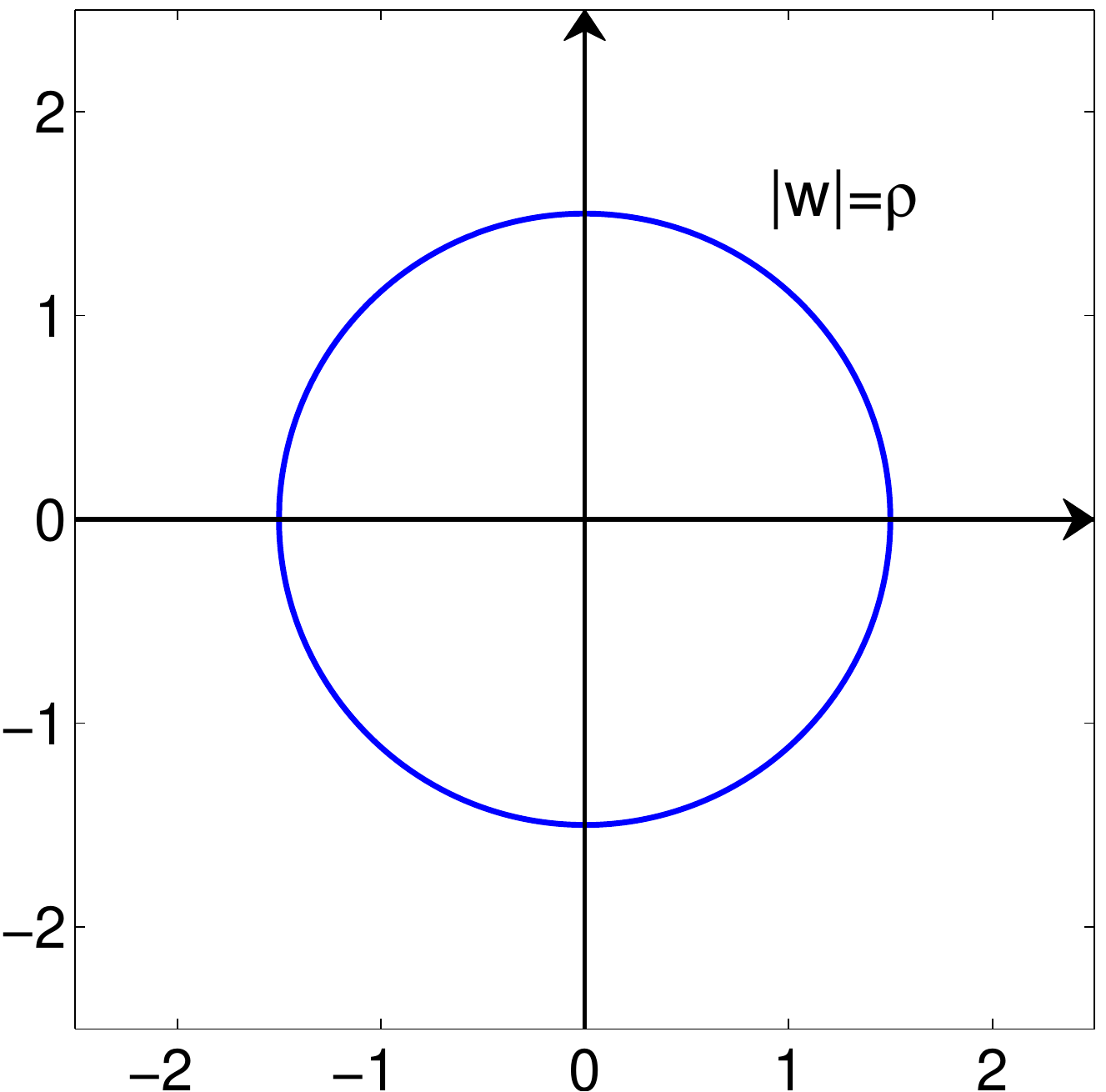}
    \includegraphics[width=.4 \textwidth]{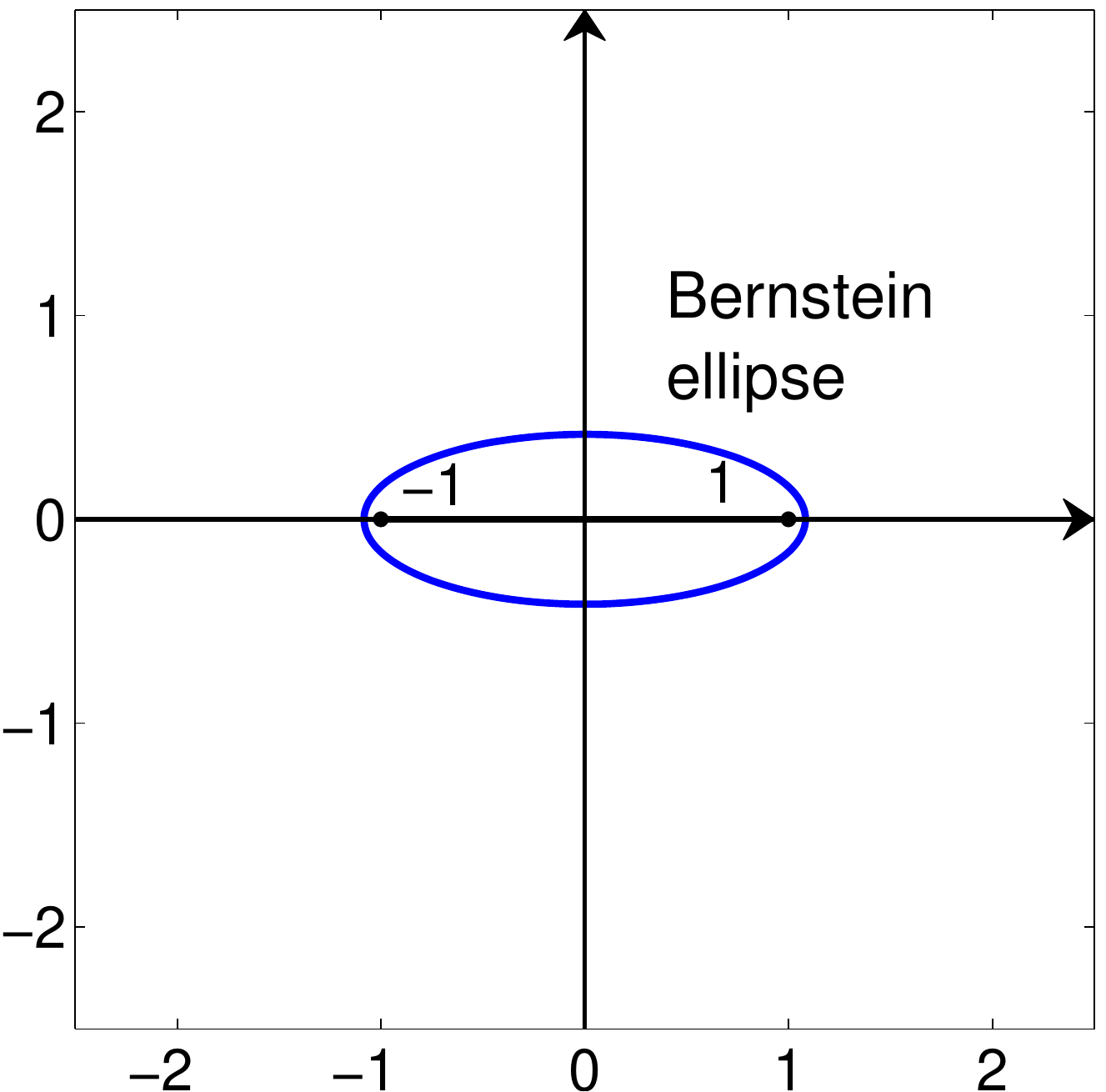}
\caption{\small Circle (left): $|w|=\rho=1.5$, and Bernstein ellipse (right):
${\mathcal E}_\rho$ with foci  $\pm1$ linked by the conformal mapping: $z=\frac 12(w+w^{-1})$.}\label{Afigrr}
  \end{center}
  \end{figure}

According to \cite{ReddyWeideman2005}, the perimeter of ${\mathcal E}_\rho$ satisfies
\begin{equation}\label{Lehp}
L({\mathcal E}_\rho) \leq\pi\sqrt{\rho^2+\rho^{-2}},
\end{equation}
which overestimates the perimeter by less than 12 percent. The distance from ${\mathcal E}_\rho$ to the interval $[-1,1]$ is
\begin{equation}\label{drho}
\delta_\rho=\frac12(\rho+\rho^{-1})-1.
\end{equation}

We are concerned with the interpolation and spectral
differentiation of analytic functions on the Gegenbauer-Gauss-type points.
Let $\{x_j:=x_j(\lambda,n)\}_{j=0}^n$ be the
Gegenbauer-Gauss points (i.e.,  the zeros
 of $C_{n+1}^{\lambda}(x)$) or the Gegenbauer-Gauss-Lobatto points
 (i.e., the zeros of $(1-x^2) C_{n-1}^{\lambda+1}(x)$).
 The associated Lagrange interpolation polynomial of $u$ is given by
 $I_{n} u\in {\mathbb P}_n$ (the set of all polynomials of degree  $\le n$)
  such that $(I_{n} u)(x_j)=u(x_j)$ for $0\le j\le n.$
  Our starting point is the Hermite's contour integral
(see, e.g., \cite{MR760629}):
\begin{equation}\label{Hermite}
(u-I_{n} u)(x)=\frac{1}{2\pi \ri}\oint_{{\mathcal E}_\rho}
\frac{Q_{n+1}(x)}{Q_{n+1}(z)}\frac{u(z)}{z-x}dz,\quad \forall x\in
[-1,1],
\end{equation}
where $Q_{n+1}(x)=C_{n+1}^{\lambda}(x)$ or $(1-x^2)C_{n-1}^{\lambda+1}(x).$
Consequently, we have
\begin{equation}\label{HermiteSP}
(u-I_{n} u)'(x_j)=\frac{1}{2\pi \ri}\oint_{\mathcal E_\rho}
\frac{Q_{n+1}'(x_j)}{Q_{n+1}(z)}\frac{u(z)}{z-x_j}dz, \quad  0\le
j\le n.
\end{equation}
 A crucial component of the error analysis is to obtain a sharp asymptotic estimate of $Q_{n+1}(z)$
on ${\mathcal E}_\rho$ with large $n.$ This will be the main concern of the forthcoming section.

\section{Asymptotic estimate of Gegenbauer polynomials on ${\mathcal E}_\rho$}


 Much of our analysis relies on the following representation of the
 Gegenbauer polynomial.
\begin{lemma}\label{lemma1.4} Let $z=\frac12(w+w^{-1}).$  We have
\begin{equation}\label{Gexp}
C_n^{\lambda}(z)=\sum_{k=0}^n
g_k^{\lambda}g_{n-k}^{\lambda}w^{n-2k},\quad n\ge 0, \;\; \lambda>-1/2,
\end{equation}
where
\begin{equation}\label{ajexp0}
g_0^{\lambda}=1,\quad g_k^{\lambda}={{k+\lambda-1}\choose
k}=\frac{\Gamma(k+\lambda)}{k!\Gamma(\lambda)},\quad 1\le k\le n.
\end{equation}
\end{lemma}
This formula is derived from the three-term recurrence formula \eqref{PropG3}
and the mathematical induction. Its proof is provided in Appendix
\ref{AppendixA}.

\begin{rem}\label{rmk:Legche} Some  consequences of Lemma \ref{lemma1.4} are in order.
\begin{itemize}
\item[(a)] Comparing the coefficients $w^n$ on both sides of
\eqref{Gexp}, we find that the leading coefficient of $C_n^\lambda$ is $2^n g_n^\lambda.$
This can be also verified from \eqref{PropG3} by the mathematical induction.

\item[(b)] If $\lambda>0,$  then $g_k^\lambda>0$ for all $0\le k\le n.$ On the other hand,
if $\lambda<0,$ we find from \eqref{gammfunP3} and \eqref{ajexp0} that
\begin{equation}\label{redneglambda}
g_k^{\lambda}=\frac{\sin(\pi\lambda)}{\pi}\frac{\Gamma{(k+\lambda)}\Gamma{(1-\lambda)}}{k!}<0,\quad 1\le k\le n.
\end{equation}
\item[(c)]  If $\lambda=1/2,$  it follows from
\eqref{gammfunP1} and \eqref{ajexp0} that
\begin{equation}\label{legendgklam}
g_k^{\lambda}=\frac{(2k)!}{(k!)^2 2^{2k}},\quad 0\le k\le n.
\end{equation}
Such a representation for the Legendre polynomial can be found in,
e.g., \cite{MR760629} and \cite{szeg75}, but the derivation is quite
different.
\item[(d)] If $\lambda=0,$ then by \eqref{chebdefn},
\begin{equation}\label{chebform}
T_n(z)=\frac 1 2 \big(w^{n}+w^{-n}\big),\quad n\ge 1.
\end{equation}
\item[(e)] If $\lambda=1,$ then $g_k^{\lambda}\equiv1$ for $0\le k\le n.$ Therefore,
the Chebyshev polynomial of the second kind has the representation
\begin{equation}\label{chbsecdkind}
 U_n(z)=\frac{w^{n+1}-w^{-(n+1)}}{w-w^{-1}}=w^n \sum_{k=0}^nw^{-2k}={C_n^{1}(z)},\quad n\ge 0,
 \end{equation}
 which can be found in  \cite{MR1937591}. \qed
\end{itemize}
\end{rem}

It is interesting to  observe  from \eqref{chbsecdkind} that for
$\lambda=1,$ $C_n^{\lambda}(z)/w^n$  converges to
$(1-w^{-2})^{-\lambda}$ uniformly for all $|w|>1,$ that is,
   \[
   \sum_{k=0}^\infty w^{-2k}=\frac 1 {1-w^{-2}},\quad |w|>1.
   \]
In what follows, we  show  a similar property holds for general $\lambda>-1/2$ and
$\lambda\not =0.$ More precisely, we  estimate the upper bound
of remainder:
\begin{equation}\label{asympto}
\begin{split}
\Big|\big(1-w^{-2}\big)^{-\lambda}-\frac{C_n^{\lambda}(z)}{g_n^{\lambda}
w^n}\Big|&\le \sum_{k=1}^n|\dnk|
|g^{\lambda}_k|\rho^{-2k}+\sum_{k=n+1}^\infty
|g_k^{\lambda}|\rho^{-2k} :=R_n(\rho,\lambda),
\end{split}
\end{equation}
where $z\in {\mathcal E}_\rho$ with $|w|=\rho>1,$ and
\begin{equation}\label{asympto2}
 \dnk=1-\frac{g_{n-k}^\lambda}{g_n^\lambda},\;\;\; 1\le k\le n.
\end{equation}

To obtain a sharp estimate of $R_n(\rho,\lambda),$ it is necessary to understand the behavior of the coefficients $\{g_k^\lambda\}_{k=1}^n$ and $\{\dnk\}_{k=1}^n,$ which are summarized in the following two lemmas.
\begin{lemma}\label{lm:asympgk} For $\lambda>-1/2,\,  k\ge 1$ and $k+\lambda\ge 1,$
\begin{equation}\label{gkasymp}
c_1\,\Big(1+\frac \lambda k \Big)^{k+1/2} e^{-\lambda} \le \frac{\Gamma(\lambda)\,g_k^{\lambda}}{(k+\lambda)^{\lambda-1}}\le c_2\,\Big(1+\frac \lambda k \Big)^{k+1/2} e^{-\lambda},
\end{equation}
where $c_1=e^{- \frac 1{12k}}$ and  $c_2=e^{\frac 1{12(k+\lambda)}}.$
\end{lemma}
\begin{proof} Applying the Stirling's formula \eqref{stirling} to
\[
(k+\lambda) \Gamma(\lambda)\, g_k^\lambda=\frac{\Gamma(k+\lambda+1)}{\Gamma(k+1)}
\]
leads to  \eqref{gkasymp}.
\end{proof}
\begin{lemma}\label{lemma1.5} Let $\{g_k^\lambda\}_{k=1}^n$ and $\{\dnk\}_{k=1}^n$  be the sequences as
defined in \eqref{ajexp0} and \eqref{asympto2}, respectively.
\begin{itemize}
\item[(i)] If  $\lambda>1$, then there holds
\begin{equation}\label{dnkcase1}
0<d_{n,1}^{\lambda}<d_{n,2}^{\lambda}< \cdots <d_{n,n}^\lambda <1.
\end{equation}
\item[(ii)] If $-1/2<\lambda<1$ and $\lambda\not =0,$  then
\begin{equation}\label{dnkcase2aa}
\cdots <|g_{k+1}^\lambda|<|g_{k}^\lambda|<\cdots <|g_{1}^\lambda|<g_{0}^\lambda=1,
\end{equation}
and we have
\begin{equation}\label{dnkcase2}
0<-d_{n,1}^{\lambda}<-d_{n,2}^{\lambda}< \cdots <-d_{n,n-1}^\lambda,
\end{equation}
and
\begin{equation}\label{dnkcase3}
 |d_{n,k}^\lambda g_k^{\lambda}|<1\;\;\;{\rm for}\;\;\;   1\le  k\le n-1,\;  n\ge 3.   
\end{equation}
\end{itemize}
\end{lemma}
\begin{proof} By \eqref{ajexp0},
\begin{equation}\label{ckprop1}
\frac{g_{k+1}^{\lambda}}{g_k^{\lambda}}= \frac{k+\lambda}{k+1}\;\;\; {\rm for}\;\; \lambda \not =0.
\end{equation}
Thus for $\lambda>1,$ $\{g_k^{\lambda}\}$ is strictly increasing  with respect to $k,$ which, together with
the fact  $g_k^\lambda>0,$ implies
\[
0<\dnk=1-\frac{g_{n-k}^\lambda}{g_{n}^\lambda}<1,
\]
and
\[
d_{n,k+1}^\lambda-\dnk= \frac{g_{n-k}^\lambda-g_{n-k-1}^\lambda}{g_{n}^\lambda}>0.
\]
This completes the proof of (i).

The property \eqref{dnkcase2aa} is a direct consequence of \eqref{ckprop1}, and \eqref{dnkcase2} can be proved in a fashion  similar to \eqref{dnkcase1}.  It remains  to verify \eqref{dnkcase3}.
If $k=1,$ a direct calculation by using \eqref{ajexp0} yields
\[
|d_{n,1}^\lambda g_1^\lambda|=|\lambda|\frac{1-\lambda}{n-(1-\lambda)}<1,\quad \forall n\ge 3.
\]
For $2\le k\le n-1,$  it follows from \eqref{ajexp0} and  \eqref{dnkcase2} that
\[
\begin{split}
|d_{n,k}^\lambda g_k^{\lambda}|&= \Big(\frac{g_{n-k}^\lambda}
{g_{n}^\lambda} -1\Big)|g_k^{\lambda}|< \frac{g_{n-k}^\lambda}
{g_{n}^\lambda}|g_k^{\lambda}| =\left(\prod_{j=0}^{k-2}
\frac{1-\frac{1-\lambda}{k-j}}{1-\frac{1-\lambda}{n-j}}\right)
\frac{|\lambda|}{1-\frac{1-\lambda}{n-k+1}}<1.
\end{split}
\]
This ends the proof.
\end{proof}

With the above preparation, we  are ready to present the main result on the upper bound of
$R_n(\rho,\lambda)$ in \eqref{asympto}.
\begin{theorem}\label{thm3.1} Let $R_n(\rho,\lambda)$ with $\rho>1$ be the remainder as defined
in  \eqref{asympto}.
\begin{itemize}
\item[(i)] If $\lambda>1,$ then for all $n\ge m\ge 1$ and
\begin{equation}\label{Aconstcond}
m+2 \ge  (\lambda-1)\Big(\frac{1}{2\ln \rho}-1\Big),
\end{equation}
we have
\begin{equation}\label{case1abest}
R_n(\rho,\lambda)\le d_{n,m}^{\lambda}
\Big((1-\rho^{-2})^{-\lambda}-1\Big)+ A\frac{[\lambda]!} {(2\ln \rho)^{\lambda}} \frac{(m+\lambda)^{[\lambda]}}{\rho^{2(m-1)}},
\end{equation}
where  $[\lambda]$ is the largest integer $\le \lambda,$ and
\begin{equation}\label{AconstAA}
 A= \frac 1 {\Gamma(\lambda)} {\rm exp}\Big({\frac 1 {12(m+1+\lambda)}+\frac \lambda {2(m+1)}}\Big).
\end{equation}
\item[(ii)]  If $-1/2<\lambda<1$ and $\lambda\not =0,$ then for all $n\ge m\ge 1$ and $n\ge 3,$
\begin{equation}\label{case1abestc}
R_n(\rho,\lambda) \le |d_{n,m}^{\lambda}| \big|(1-\rho^{-2})^{-\lambda}-1\big|+\frac{\rho^{-2m}}{\rho^2-1}+2 \rho^{-2n}.
\end{equation}
\end{itemize}
Here, the factor $d_{n,m}^{\lambda}$ is given by \eqref{asympto2}.
\end{theorem}
\begin{proof}  (i) For $\lambda>1,$ we obtain from \eqref{dnkcase1} that
\begin{equation}\label{case1a}
\begin{split}
R_n(\rho,\lambda)&=\sum_{k=1}^m \dnk g_k^{\lambda}\rho^{-2k}+\sum_{k=m+1}^n \dnk g_k^{\lambda}\rho^{-2k}
+\sum_{k=n+1}^\infty  g_k^{\lambda}\rho^{-2k}\\
&\overset{\re{dnkcase1}}\leq d_{n,m}^{\lambda} \sum_{k=1}^m g_k^{\lambda}\rho^{-2k}+\sum_{k=m+1}^n g_k^{\lambda} \rho^{-2k} +\sum_{k=n+1}^\infty  g_k^{\lambda}\rho^{-2k}
\\
&\leq d_{n,m}^{\lambda}
\Big((1-\rho^{-2})^{-\lambda}-1\Big)+ \sum_{k=m+1}^\infty g_k^{\lambda} \rho^{-2k}.
\end{split}
\end{equation}
By Lemma \ref{lm:asympgk},
\begin{equation*}\label{case1abc}
\begin{split}
\sum_{k=m+1}^\infty g_k^{\lambda} \rho^{-2k}&\le {(\Gamma(\lambda))^{-1}}    {e^{-\lambda}e^{\frac 1 {12(m+1+\lambda)}}} \sum_{k=m+1}^\infty
\frac{(k+\lambda)^{\lambda-1}}{\rho^{2k}} \Big(1+\frac \lambda k\Big)^{k+1/2}
\\&\le {(\Gamma(\lambda))^{-1}} e^{-\lambda} e^{\frac 1 {12(m+1+\lambda)}}  \sum_{k=m+1}^\infty
\frac{(k+\lambda)^{\lambda-1}}{\rho^{2k}}  e^{\lambda+\frac \lambda {2k}},
\end{split}
\end{equation*}
where we used the inequality $1+x<e^x$ for $x>0.$ Hence,
\begin{equation}\label{Aconst}
\sum_{k=m+1}^\infty g_k^{\lambda} \rho^{-2k}\le A \sum_{k=m+1}^\infty \frac{(k+\lambda)^{\lambda-1}}{\rho^{2k}},
\end{equation}
where $A$ is given by \eqref{AconstAA}.
One verifies that under the condition \eqref{Aconstcond},  ${(k+\lambda)^{\lambda-1}}/{\rho^{2k}}$ is decreasing with respect to $k.$ Therefore, by \eqref{incompleteGamma} and \eqref{incompleteGammaP1},
\begin{equation*}\label{temp0}
\begin{split}
\sum_{k=m+1}^\infty \frac{(k+\lambda)^{\lambda-1}}{\rho^{2k}}&\le \int_m^\infty {(x+\lambda)^{\lambda-1}}\rho^{-2x} dx=\frac{\rho^{2\lambda}}{(2\ln \rho)^{\lambda}} \int_{2(m+\lambda)\ln \rho}^\infty x^{\lambda-1} e^{-x} dx\\
&=\frac{\rho^{2\lambda}}{(2\ln \rho)^{\lambda}} \Gamma\big(\lambda, 2(m+\lambda)\ln \rho\big)
\le
\frac{\rho^{2\lambda}}{(2\ln \rho)^{\lambda}} \Gamma\big([\lambda]+1, 2(m+\lambda)\ln \rho\big)\\
&=\frac{[\lambda]!  \rho^{-2m}}{(2\ln \rho)^{\lambda}}
\sum_{k=0}^{[\lambda]} \frac{(m+\lambda)^k(2\ln \rho)^k}{k!}\le \frac{[\lambda]!}{(2\ln \rho)^{\lambda}} \frac{(m+\lambda)^{[\lambda]}}{\rho^{2m}} \sum_{k=0}^{\infty} \frac{(2\ln \rho)^k}{k!}  \\
&= \frac{[\lambda]!} {(2\ln \rho)^{\lambda}}
\frac{(m+\lambda)^{[\lambda]}}{\rho^{2(m-1)}}.
\end{split}
\end{equation*}
 A combination of the above estimates leads to \eqref{case1abest}.

(ii) Now, we turn to the proof of the second case:  $-1/2<\lambda<1$ and $\lambda\not =0.$  By Lemma
\ref{lemma1.5},
\begin{equation*}\label{case2a}
\begin{split}
R_n(\rho,\lambda)&=\sum_{k=1}^m |\dnk| |g_k^{\lambda}|\rho^{-2k}+\sum_{k=m+1}^{n}|\dnk||g_k^{\lambda}|\rho^{-2k}
+\sum_{k=n+1}^\infty|g_k^\lambda|\rho^{-2k} \\
&\overset{\re{dnkcase2}}\leq |d_{n,m}^{\lambda}|\sum_{k=1}^m |g_k^{\lambda}|\rho^{-2k}
\overset{\re{dnkcase3}}+\sum_{k=m+1}^{n-1}\rho^{-2k}+|d_{n,n}^{\lambda}g_n^{\lambda}|\rho^{-2n} \overset{\re{dnkcase2aa}}+ \sum_{k=n+1}^\infty \rho^{-2k}\\
&\quad \leq |d_{n,m}^{\lambda}| \big|(1-\rho^{-2})^{-\lambda}-1\big|+\frac{\rho^{-2m}}{\rho^2-1}+2 \rho^{-2n},
\end{split}
\end{equation*}
where in the last step, we used the following facts:
\[
\begin{split}
\sum_{k=1}^m |g_k^{\lambda}|\rho^{-2k}&\leq{\rm sign}(\lambda)
\sum_{k=1}^\infty g_k^{\lambda}\rho^{-2k}={\rm sign}(\lambda)
\big((1-\rho^{-2})^{-\lambda}-1\big),\quad \rho>1,
\end{split}
\]
(note: ${\rm sign}(\lambda)$ is the sign of $\lambda$), and
$|d_{n,n}^{\lambda}g_n^{\lambda}|=|g_n^{\lambda}-1|<2,$ thanks to  \eqref{dnkcase2aa}.
\end{proof}

The estimate in Theorem \ref{thm3.1} is quite tight and is valid even for small $n.$  By choosing
a suitable  $m$ to balance the two error terms in the upper bound, we are able to derive the anticipated  asymptotic estimate.
\begin{theorem}\label{Th:exactest} For any   $z\in {\mathcal E}_\rho$ with $|w|=\rho>1,$ and
any $\lambda>-1/2$ and $\lambda\not =0,$ there exists
$0<\varepsilon\leq1/2$ such that
\begin{equation}\label{mainasymestimate}
\Big|\big(1-w^{-2}\big)^{-\lambda}-\frac{C_n^{\lambda}(z)}{g_n^{\lambda}
w^n}\Big|\le A(\rho,\lambda)n^{\varepsilon-1}+O(n^{-1}),
\end{equation}
where
\begin{equation}\label{Arhoconst}
A(\rho,\lambda)=|1-\lambda| \big|(1-\rho^{-2})^{-\lambda}-1\big|.
\end{equation}
\end{theorem}
\begin{proof}
We first estimate $|d_{n,m}^\lambda|$ in Theorem \ref{thm3.1}, when  $n-m$ is large.
Using the Stirling's formula  \eqref{stirling} and \eqref{ajexp0} that
\begin{equation*}
\begin{split}
\frac{g_{n-m}^\lambda}{ g_n^{\lambda}}&=\Big(1+\frac{1-\lambda}{n+\lambda-1}\Big)^{n+\frac 1 2}\Big(1-\frac{1-\lambda}{n-m}\Big)^{n-m+\frac 12}
\Big(1-\frac{m}{n+\lambda-1}\Big)^{\lambda-1}\Big\{1+O\Big(\frac 1 {n-m}\Big)\Big\}\\
&=\Big(1-\frac{m}{n+\lambda-1}\Big)^{\lambda-1}\Big\{1+O\Big(\frac 1 {n-m}\Big)\Big\}\\
&=\Big\{1+\frac{(1-\lambda)m}{n+\lambda-1}+O\Big(\frac{m^2}{n^2}\Big)\Big\}\Big\{1+O\Big(\frac 1 {n-m}\Big)\Big\}.
\end{split}
\end{equation*}
Hereafter, taking  $m=[n^\varepsilon]$ with $0<\varepsilon\leq1/2$
yields
\begin{equation}\label{gratioest}
\begin{split}
\frac{g_{n-m}^\lambda}{ g_n^{\lambda}}&=1+(1-\lambda) n^{\varepsilon-1}+O\Big(\frac 1{n-n^\varepsilon} \Big)\\
&\Longrightarrow\;\;
d_{n,m}^\lambda= (\lambda-1) n^{\varepsilon-1}+O(n^{-1}).
\end{split}
\end{equation}
One verifies readily that  for $\lambda>1$ and any $0<\varepsilon
\leq1/2,$
\begin{equation}\label{mlmabda}
\frac{m^{[\lambda]}}{\rho^{2m}}\le \frac 1 n\;\; \Longleftrightarrow \;\; \frac{\ln n}{n^\varepsilon} \le
\frac{2\ln \rho} {1+\varepsilon [\lambda]},
\end{equation}
which, together with \eqref{case1abest} and \eqref{gratioest}, implies
\eqref{mainasymestimate} with $\lambda>1.$

If $-1/2<\lambda<1$ and $\lambda\not=0,$ it follows from \eqref{mlmabda}  that $\rho^{-2m}\le n^{-1}$ for any
 $0<\varepsilon \leq 1/2.$ This validates the desired estimate.
\end{proof}



A direct consequence of Theorem \ref{Th:exactest} is that
\begin{equation}\label{cnlamdaz}
\lim_{n\to\infty} \frac{C_n^\lambda(z)}{g_n^{\lambda}}=
\lim_{n\rightarrow\infty}\sum_{k=0}^n \frac { g_{n-k}^\lambda }
{g_n^\lambda} g_k^\lambda w^{-2k}=(1-w^{-2})^{-\lambda},
\end{equation}
for all $z\in {\mathcal E}_\rho$ with $|w|=\rho>1,$ and any $\lambda>-1/2$ and $\lambda\not=0.$

\begin{rem}\label{Augaddwang}  Based on a completely different argument,
Elliott  \cite{Elliott71} derived an asymptotic expansion  for large $n$ near $z=1$ (but not near $z=-1$):
$C_n^\lambda(z)\sim \frac{B(n,\lambda)}{(z^2-1)^{\lambda/2}},$ where $B$ is a series involving modified Bessel functions, and some other asymptotic expansions for $|z|$ large and $n$ fixed.  Although they are valid for general $z$ off the interval $[-1,1],$ our results in Theorems \ref{thm3.1} and  \ref{Th:exactest} provide tighter and sharper bounds when $z$ is sitting on ${\mathcal E}_\rho.$
 \qed
\end{rem}

As the end of this section, we provide some numerical results to illustrate the tightness of the upper bound in \eqref{mainasymestimate}.  Denote by
\begin{equation}\label{maxerrbnd}
E_n(\rho;\lambda):=\frac 1 {A(\rho,\lambda)}\max_{z\in {\mathcal E}_\rho} \Big|\big(1-w^{-2}\big)^{-\lambda}-\frac{C_n^{\lambda}(z)}{g_n^{\lambda}
w^n}\Big|.
\end{equation}
To approximate the maximum value, we sample a set of points dense on
the ellipse ${\mathcal E}_\rho$ based on  the conformal mapping
$z=\frac 1 {2}(w+w^{-1})$ of the Fourier points on the circle
$w=\rho e^{\ri\theta}.$ We plot in Figure \ref{figbnd} in (Matlab) log-log scale of
 $E_n(\rho;\lambda), n^{-1}$ and $n^{\varepsilon-1}$ (with $\varepsilon=0.1$) for
 several sets of parameters $\lambda$ and $\rho,$ and for large $n.$
 According to Theorem
\ref{Th:exactest},  $E_n$ should be bounded by $n^{\varepsilon-1}$ from above, and it is anticipated to be bounded below by $n^{-1},$ if the estimate is tight. Indeed, we observe from Figure \ref{figbnd}
such a behavior when $n$ is large.

\begin{figure}[!th]
  \begin{center}
    \includegraphics[width=.45 \textwidth]{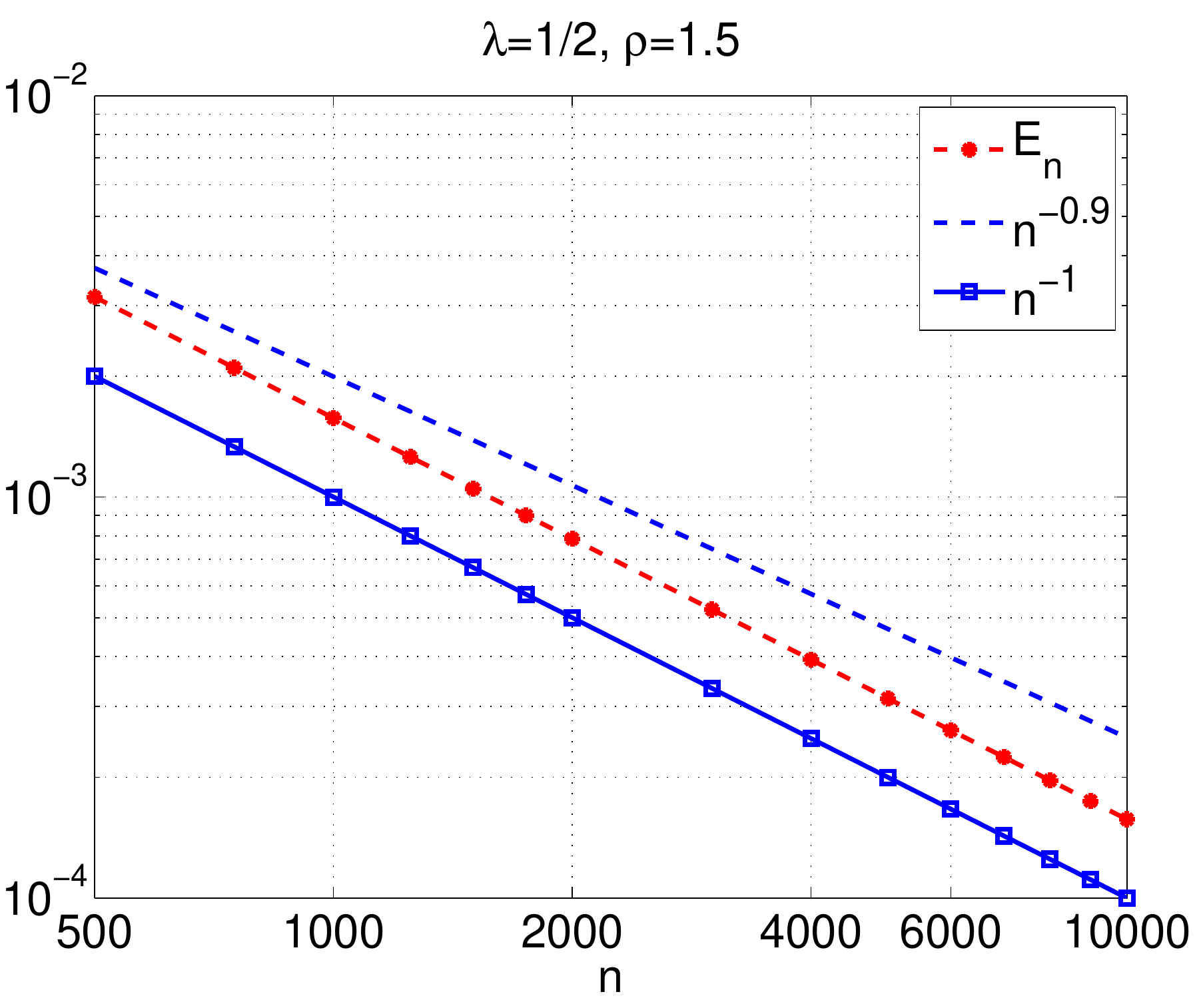} \includegraphics[width=.45 \textwidth]{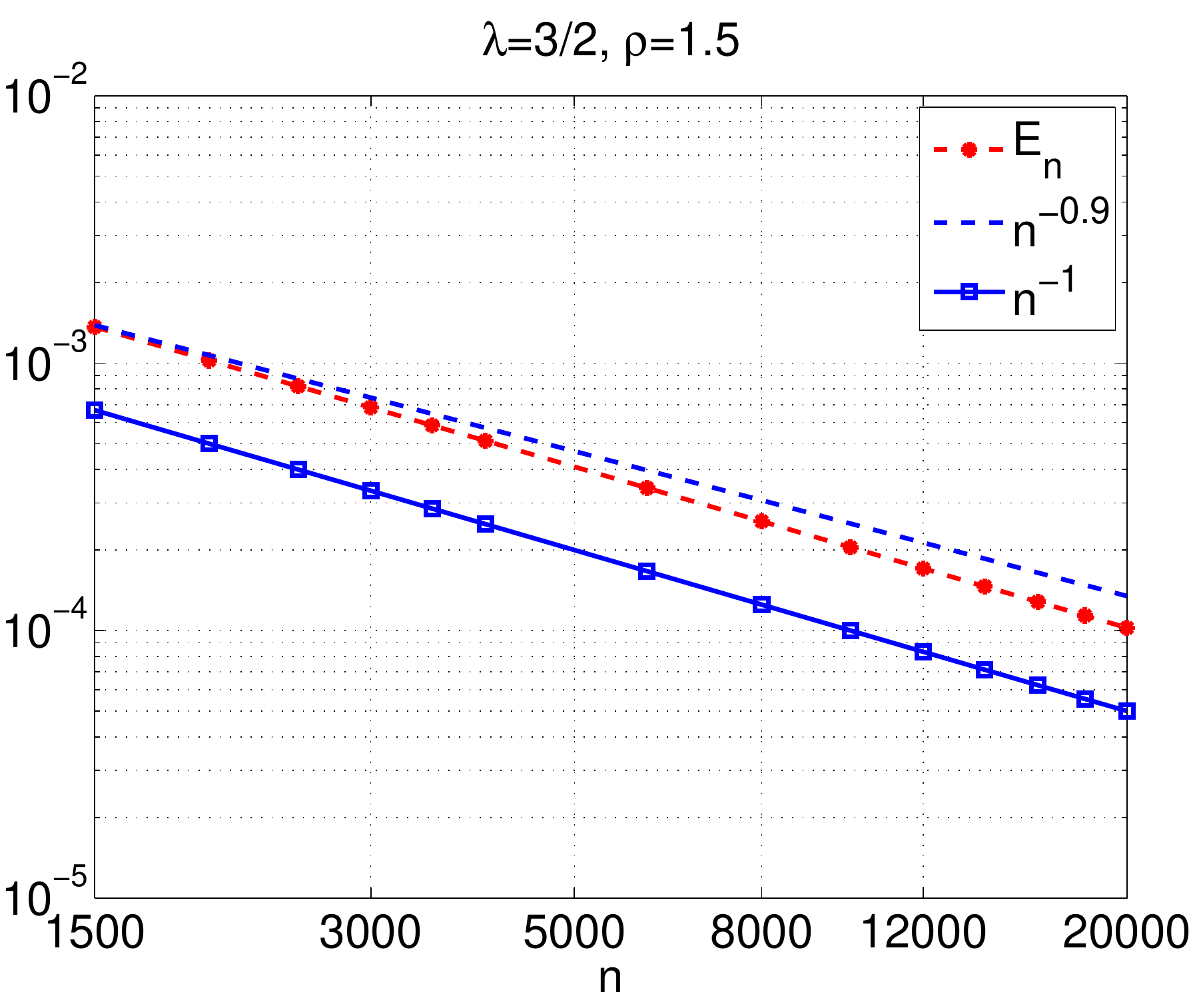}\\
    \includegraphics[width=.45 \textwidth]{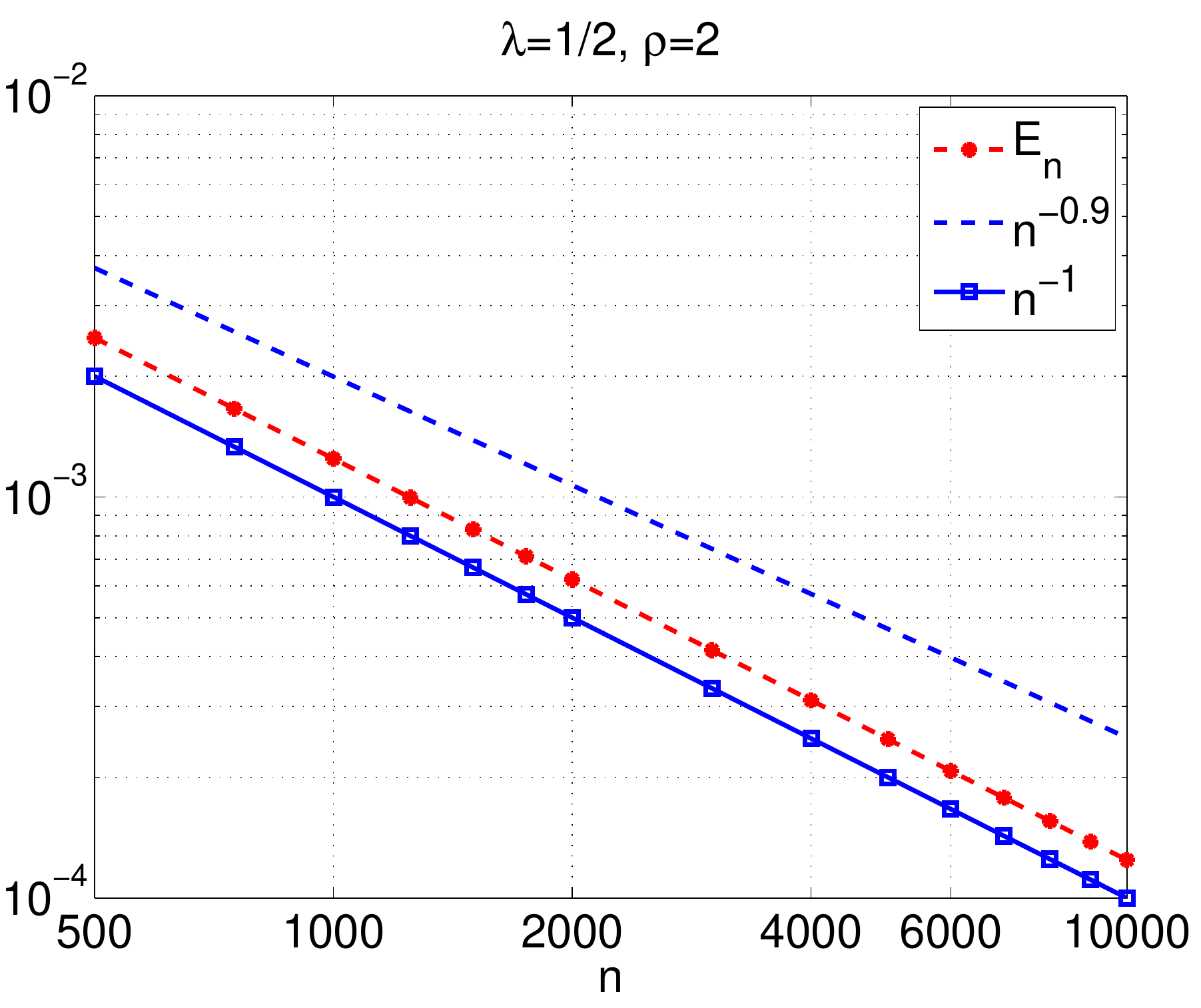} \includegraphics[width=.45 \textwidth]{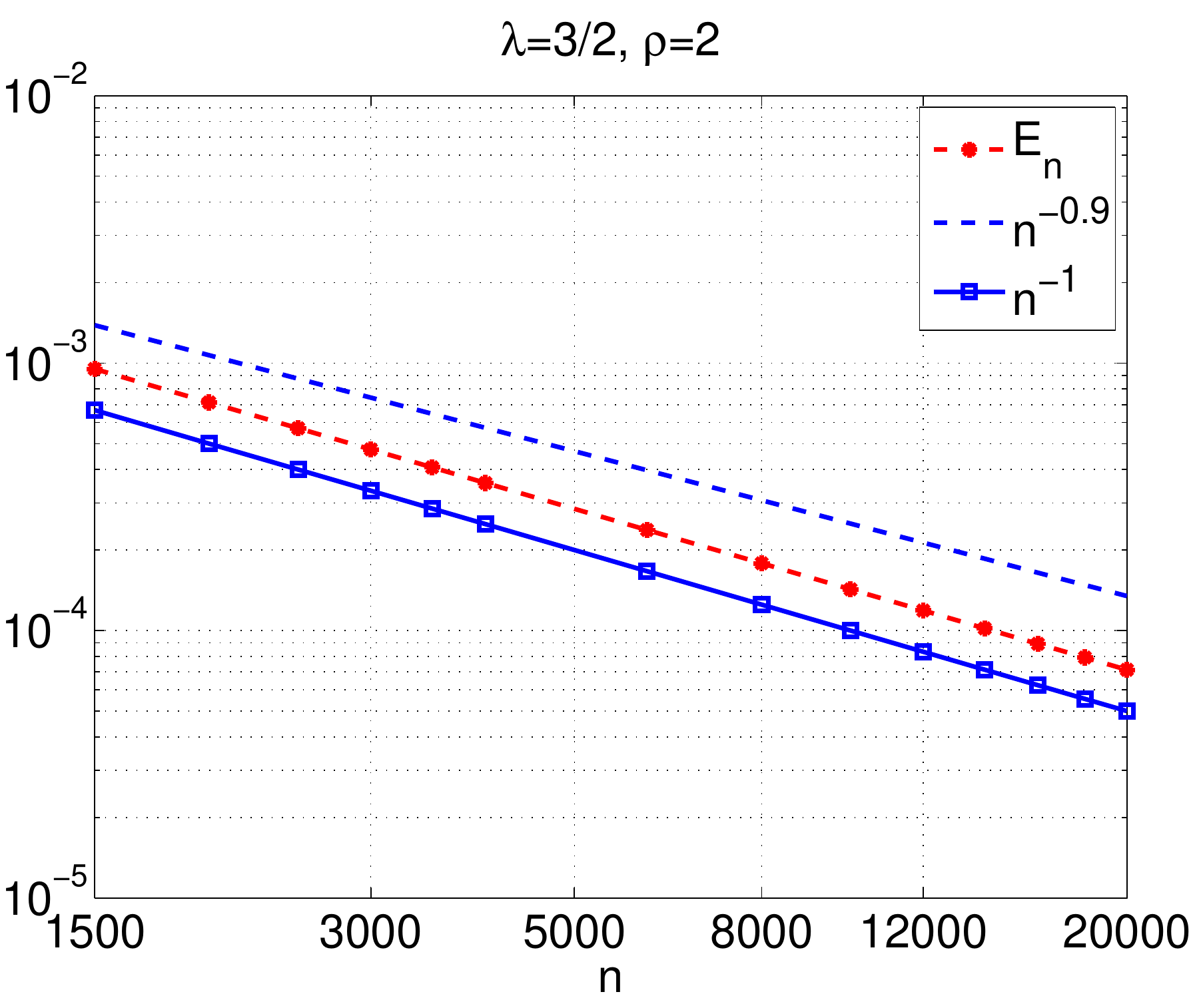}\\
\caption{\small $E_n$ against $n^{\varepsilon-1}$ (with $\varepsilon=0.1$) and
$n^{-1}$ for large $n$.}\label{figbnd}
  \end{center}
  \end{figure}

\section{Error estimates of interpolation and spectral differentiation}

After collecting all the necessary results, we are ready to estimate exponential convergence of interpolation
and spectral differentiation of analytic functions.

Hereafter,  the notation  $a_n\cong b_n$ means that  $a_n/b_n\to 1$ as $n\to\infty,$
for any two sequences $\{a_n\}$ and $\{b_n\}$ (with $b_n\not =0$) of complex numbers.

\subsection{Gegenbauer-Gauss interpolation and differentiation}
We start with the analysis of interpolation and spectral differentiation  on  zeros of the
Gegenbauer polynomial $C_{n+1}^\lambda(x).$
\begin{thm}\label{theorem3.1}
Let  $u$ be analytic on and within the ellipse ${\mathcal E_\rho}$
with foci $\pm 1$ and $\rho>1$ as defined in \eqref{Berellips}, and
let $(I_nu)(x)$ be the interpolant of $u(x)$ at the set of $(n+1)$
Gegenbauer-Gauss points.
\begin{itemize}
\item[(i)] If $\lambda>0,$ we have
\begin{equation}\label{intpoestcase1}
\max_{|x|\le 1}\big|(u-I_{n} u)(x)\big|  \le  \frac{c
\Gamma(\lambda)M_{\rho}\sqrt{\rho^2+\rho^{-2}}}{\Gamma(2\lambda)(\rho-1)^2(1+\rho^{-2})^{-\lambda}}
\frac{n^\lambda} {\rho^n}.
\end{equation}
\item[(ii)] If $-1/2<\lambda<0,$ we have
\begin{equation}\label{intpoestcase2}
\max_{|x|\le 1}\big|(u-I_{n} u)(x)\big|  \le  \frac{c D_\lambda
|\Gamma(\lambda)|M_{\rho}\sqrt{\rho^2+\rho^{-2}}}{(\rho-1)^2(1-\rho^{-2})^{-\lambda}}
\frac{1} {\rho^n}.
\end{equation}
\end{itemize}
Here, $M_{\rho}=\max_{z\in{\mathcal E_\rho}}|u(z)|$, $D_\lambda$  is defined in \eqref{Gegenmaxim}, and
$c\cong 1$ is a generic positive constant.
\end{thm}
\begin{proof} By the formula \eqref{Hermite} with $Q_{n+1}=C_{n+1}^\lambda$ and
\eqref{Lehp}-\eqref{drho}, we have the bound  of the point-wise error:
\begin{equation}\label{HermiteSP10}
\begin{split}
\big|(u-I_{n} u)(x)\big|&\le \frac{|C_{n+1}^\lambda(x)|}{2\pi}
\frac{\max_{z\in \mathcal E_\rho}|u(z)|}{\min_{z\in{\mathcal
E_\rho}}|C_{n+1}^\lambda (z)|}\oint_{\mathcal E_\rho}
\frac{|dz|}{|z-x|}
\\&
\leq \frac{M_{\rho} L({\mathcal E_\rho}) }{2\pi \delta_\rho}
\frac{|C_{n+1}^\lambda(x)|}
{\min_{z\in{\mathcal E_\rho}}|C_{n+1}^\lambda(z)|}\\
&\le \frac{M_{\rho}\sqrt{\rho^2+\rho^{-2}}}{\rho+\rho^{-1}-2}
\frac{|C_{n+1}^\lambda(x)|}
{\min_{z\in{\mathcal E_\rho}}|C_{n+1}^\lambda(z)|},\quad x\in[-1,1],\; n\ge 0.
\end{split}
\end{equation}
Therefore, it is essential to obtain the lower bound of
$|C_{n+1}^\lambda(z)|.$ Recall that for any two complex numbers
$z_1$ and $z_2,$ we have $\big||z_1|-|z_2|\big|\le |z_1-z_2|.$ It
follows from Theorem \ref{Th:exactest} that
\begin{equation*}
\bigg||1-w^{-2}|^{-\lambda}-\frac{|C_{n+1}^{\lambda}(z)|}{|g_{n+1}^{\lambda}|
\rho^{n+1}}\bigg|\le A(\rho,\lambda)n^{\varepsilon-1}+O(n^{-1}),
\end{equation*}
 which implies
\begin{equation}\label{Cgenbound}
\begin{split}
\big|1-w^{-2}\big|^{-\lambda}&-A(\rho,\lambda)n^{\varepsilon-1}-O(n^{-1})
\le \frac{|C_{n+1}^{\lambda}(z)|}{|g_{n+1}^{\lambda}|
\rho^{n+1}}\\
&\le
\big|1-w^{-2}\big|^{-\lambda}+A(\rho,\lambda)n^{\varepsilon-1}+O(n^{-1}).
\end{split}
\end{equation}
Notice that
\begin{equation}\label{wrhobnd}
1-\rho^{-2}\leq|1-w^{-2}|\leq 1+\rho^{-2}.
\end{equation}
Consequently,
\begin{equation}\label{newconse}
|C_{n+1}^\lambda(z)|\ge
c\frac{n^{\lambda-1}\rho^{n+1}}{|\Gamma(\lambda)|}
\begin{cases}
(1+\rho^{-2})^{-\lambda},\quad & {\rm if}\;\; \lambda>0,\\
(1-\rho^{-2})^{-\lambda},\quad & {\rm if}\;\; \lambda<0,
  \end{cases}
\end{equation}
where we used \eqref{gkasymp}, and  the constant $c\cong 1.$

On the other hand, we derive from \eqref{PropG1}, \eqref{Gegenmaxim} and \eqref{stirling} that if $\lambda>0,$
\begin{equation}\label{genCmaxest}
\max_{|x|\le 1}|C_{n+1}^\lambda(x)|=C_{n+1}^\lambda(1)\cong \frac{n^{2\lambda-1}}{\Gamma(2\lambda)}.
\end{equation}
Hence,  a combination of \eqref{HermiteSP10}, \eqref{newconse} and
\eqref{genCmaxest} leads to \eqref{intpoestcase1}.  Similarly, for
$-1/2<\lambda<0,$ we use \eqref{Gegenmaxim} to derive
\eqref{intpoestcase2}.
\end{proof}

\begin{rem}\label{ptwisecom} For  $\lambda>0,$ we obtain from \eqref{gammafd}, \eqref{sharpest} and \eqref{stirling} that
\begin{equation}\label{pointmax}
|C_{n+1}^\lambda(x)|\le \frac{c 2^{1-\lambda}\sqrt{e(2+\sqrt 2 \lambda)}}{\Gamma(\lambda)} n^{\lambda-1} (1-x^2)^{-\lambda/2},\;\; |x|<1.
\end{equation}
Replacing \eqref{genCmaxest} by this bound in the above proof,  we can derive the point-wise estimate for $\lambda>0:$
\begin{equation}\label{pintgauss}
\big|(u-I_{n} u)(x)\big|\le D(\rho, \lambda)
\frac{(1-x^2)^{-\lambda/2}}{\rho^n}, \quad |x|<1,
\end{equation}
where  the positive constant $D(\rho, \lambda)$ can be worked out as
well. It appears to be sharper than \eqref{intpoestcase1} at the
points which are not too close to the endpoints $x=\pm 1.$ A similar
remark also applies to the Gegenbauer-Gauss-Lobatto interpolation to
be addressed in a minute. \qed
\end{rem}


Now, we turn to the estimate of spectral differentiation.
\begin{thm}\label{Th:spectraldiff}
Let  $u$ be analytic on and within the ellipse ${\mathcal E_\rho}$
with foci $\pm 1$ and $\rho>1$ as defined in \eqref{Berellips}, and
let $(I_nu)(x)$ be the interpolant of $u(x)$ at $(n+1)$
Gegenbauer-Gauss points $\{x_j\}_{j=0}^n$. Then we have
\begin{equation}\label{diffestest}
\max_{0\le j\le n}\big|(u-I_{n} u)'(x_j)\big| \le
\Lambda(\rho,\lambda) \frac{n^{\lambda+2}} {\rho^n},
\end{equation}
where the constant
\begin{equation}\label{constTh}
\Lambda(\rho,\lambda)=\frac{2c
\Gamma(\lambda+1)M_{\rho}\sqrt{\rho^2+\rho^{-2}}}{\Gamma(2\lambda+2)(\rho-1)^2}
 \begin{cases}
 (1+\rho^{-2})^{\lambda},\quad & {\rm if}\;\; \lambda>0,\\
 (1-\rho^{-2})^{\lambda},\quad & {\rm if}\;\; \lambda<0,
  \end{cases}
\end{equation}
and $c, M_\rho$ are the same as in {\rm Theorem  \ref{theorem3.1}}.
\end{thm}
\begin{proof} In view of \eqref{Hermite} and \eqref{HermiteSP}, it is enough to replace $x$ and
$C_{n+1}^\lambda(x)$ by $x_j$ and $\frac d {dx} C_{n+1}^\lambda(x),$ respectively, in \eqref{HermiteSP10}. Thus, we have
\begin{equation}\label{HermiteSP10a}
\begin{split}
\big|(u-I_{n} u)'(x_j)\big|&\le \frac {M_\rho}
{2\pi}
\frac{|(C_{n+1}^\lambda)'(x_j)|}{\min_{z\in\mathcal E_\rho}|C_{n+1}^\lambda (z)|}\oint_{\mathcal E_\rho} \frac{|dz|}{|z-x_j|}\\
&\le \frac{M_{\rho}\sqrt{\rho^2+\rho^{-2}}}{\rho+\rho^{-1}-2}
\frac{|(C_{n+1}^\lambda)'(x_j)|}
{\min_{z\in{\mathcal E_\rho}}|C_{n+1}^\lambda(z)|},\quad 0\le j\le n.
\end{split}
\end{equation}
By \eqref{PropG2} and \eqref{genCmaxest},
\begin{equation}\label{ddxGenmax}
\max_{|x|\le
1}\big|(C_{n+1}^\lambda)'(x)\big|=2|\lambda||C_n^{\lambda+1}(1)|\cong
\frac{2|\lambda|}{\Gamma(2\lambda+2)} n^{2\lambda+1},
\end{equation}
which, together with \eqref{newconse} and \eqref{HermiteSP10a}, leads to the desired estimate.
\end{proof}

\begin{rem}\label{highorderderiv} Obviously, by \eqref{Hermite},
\begin{equation}\label{Hermiteaa}
(u-I_{n} u)'(x)=\frac{1}{2\pi \ri}\oint_{{\mathcal E}_\rho} \Big(
\frac{(C_{n+1}^\lambda)'(x)}{z-x}+\frac{C_{n+1}^\lambda(x)}{(z-x)^2}\Big)\frac{u(z)}{C_{n+1}^\lambda(z) } dz.
\end{equation}
If $x\not= x_j,$ we need to estimate the second term in the summation, which can be done in the same fashion in the proof of Theorem \ref{theorem3.1}. The first term  is actually estimated above.  Consequently, we have
\[
\max_{|x|\le 1}\big|(u-I_{n} u)'(x)\big|\le  \Lambda(\rho,\lambda) \frac{n^{\lambda+2}} {\rho^n}+\frac {1}{\delta_\rho}
\max_{|x|\le 1}\big|(u-I_{n} u)(x)\big|,
\]
where $\delta_\rho$ is given by \eqref{drho}.
Hence, by  Theorem \ref{theorem3.1},
\begin{equation}\label{diffestestaaa}
\max_{|x|\le 1}\big|(u-I_{n} u)'(x)\big| = O\Big(\frac{n^{\lambda+2}} {\rho^n}\Big).
\end{equation}
In fact, the results for  higher-order derivatives can be derived recursively, and it is anticipated that
\begin{equation}\label{diffestestaaab}
\max_{|x|\le 1}\big|(u-I_{n} u)^{(k)}(x)\big| = O\Big(\frac{n^{\lambda+2k}} {\rho^n}\Big),\quad k\ge 1.
\end{equation}
A similar remark applies to the Gegenbauer-Gauss-Lobatto case below. \qed
 \end{rem}

\subsection{Gegenbauer-Gauss-Lobatto interpolation and differentiation}
We are now in a position to estimate the Gegenbauer-Gauss-Lobatto interpolation and spectral differentiation. In this case,  $Q_{n+1}(x)=(1-x^2)C_{n-1}^{\lambda+1}(x)$ in  \eqref{Hermite}-\eqref{HermiteSP}.  The main result is stated as follows.
\begin{thm}\label{theorem3.3}
Let  $u$ be analytic on and within the ellipse ${\mathcal E_\rho}$
with foci $\pm 1$ and $\rho>1$ as defined in \eqref{Berellips}, and
let $(I_nu)(x)$ be the interpolant of $u(x)$ at the set of $(n+1)$
Gegenbauer-Gauss-Lobatto points.
\begin{itemize}
\item[(a)] We have the interpolation error:
\begin{equation}\label{Hermitepwerr}
\max_{|x|\le 1}\big|(u-I_{n} u)(x)\big|  \le  \frac{4c
M_{\rho}\sqrt{\rho^2+\rho^{-2}}(1+\rho^{-2})^{\lambda+1}}{(1-\rho^{-1})^2(\rho-\rho^{-1})^2}\frac{\Gamma(\lambda+1)}{\Gamma(2\lambda+2)}
\frac{n^{\lambda+1}} {\rho^{n}}.
\end{equation}
\item[(b)] We have the estimate:
\begin{equation}\label{Hermiteesperr}
\begin{split}
&\max_{0\le j\le n}\big|(u-I_{n} u)'(x_j)\big| \le  \frac{8c
M_{\rho}\sqrt{\rho^2+\rho^{-2}}(1+\rho^{-2})^{\lambda+1}}{(1-\rho^{-1})^2(\rho-\rho^{-1})^2}
\frac{\Gamma(\lambda+2)} {\Gamma(2\lambda+4)}\frac{n^{\lambda+3}}
{\rho^{n}}.
\end{split}
\end{equation}
\end{itemize}
Here,  $c\cong 1$  and  $ M_\rho=\max_{z\in {\mathcal E}_\rho}|u(z)|.$
\end{thm}
\begin{proof} For any $z\in {\mathcal E}_\rho,$  one  verifies that
\[
\frac 14(\rho-\rho^{-1})^2\leq |z^2-1|\leq\frac 14(\rho+\rho^{-1})^2,
\]
and
\begin{equation}\label{minCestGL}
\min_{z\in{\mathcal
E_\rho}}\big|(1-z^2)C_{n-1}^{\lambda+1}(z)\big|\geq \frac
14(\rho-\rho^{-1})^2\min_{z\in{\mathcal
E_\rho}}\big|C_{n-1}^{\lambda+1}(z)\big|.
\end{equation}

(a)~ By  \eqref{Hermite} with
$Q_{n+1}(x)=(1-x^2)C_{n-1}^{\lambda+1}(x)$ and
\eqref{Lehp}-\eqref{drho}, we have the bound  of the point-wise
error:
\begin{equation}\label{HermiteSPGGL1}
\begin{split}
\big|(u-I_{n} u)(x)\big|&\le
\frac{\big|(1-x^2)C_{n-1}^{\lambda+1}(x)\big|}{2\pi}
\frac{\max_{z\in \mathcal E_\rho}|u(z)|}{\min_{z\in{\mathcal
E_\rho}}\big|(1-z^2)C_{n-1}^{\lambda+1}(z)\big|}\oint_{\mathcal
E_\rho}
\frac{|dz|}{|z-x|}\\
&\overset{\re{minCestGL}}\le
\frac{M_{\rho}\sqrt{\rho^2+\rho^{-2}}}{\rho+\rho^{-1}-2}
\frac{4(\rho-\rho^{-1})^{-2}|C_{n-1}^{\lambda+1}(x)|}
{\min_{z\in{\mathcal E_\rho}}|C_{n-1}^{\lambda+1}(z)|},\quad
x\in[-1,1],\; n\ge 0.
\end{split}
\end{equation}
Thus, the estimate \eqref{Hermitepwerr} follows from  \eqref{newconse} and \eqref{genCmaxest}.

(b)~Similarly, we have
\begin{equation}\label{HermiteSPGGL3}
\begin{split}
\big|(u-I_{n} u)'(x_j)\big|&\le \frac {M_\rho} {2\pi}
\frac{\big|[(1-x^2)C_{n-1}^{\lambda+1}(x)]'(x_j)\big|}{\min_{z\in\mathcal
E_\rho}\big|(1-z^2)C_{n-1}^{\lambda+1}(z)\big|}\oint_{\mathcal
E_\rho}
\frac{|dz|}{|z-x_j|}\\
&\le \frac{4M_{\rho}\sqrt{\rho^2+\rho^{-2}}}{\rho+\rho^{-1}-2}
\frac{\big|[(1-x^2)C_{n-1}^{\lambda+1}(x)]'(x_j)\big|}
{(\rho-\rho^{-1})^2\min_{z\in{\mathcal
E_\rho}}\big|C_{n-1}^{\lambda+1}(z)\big|},\quad 0\le j\le n.
\end{split}
\end{equation}
A direct calculation leads to
\begin{equation}\label{Lobattod}
\big[(1-x^2)C_{n-1}^{\lambda+1}(x)\big]'=-2xC_{n-1}^{\lambda+1}(x)+(1-x^2)\big[C_{n-1}^{\lambda+1}(x)\big]',
\end{equation}
which, together with \eqref{genCmaxest} and \eqref{ddxGenmax}, gives
\begin{equation}\label{HermiteSPGGL4}
\max\limits_{|x|\leq
1}\big|[(1-x^2)C_{n-1}^{\lambda+1}(x)]'\big|\cong
\frac{2(\lambda+1)}{\Gamma(2\lambda+4)}n^{2\lambda+3}.
\end{equation}
A combination of \eqref{newconse}, \eqref{HermiteSPGGL3} and
\eqref{HermiteSPGGL4} yields the desired estimate.
\end{proof}

\begin{rem}\label{GGLrem}  Similar to  Remarks \ref{ptwisecom}   and  \ref{highorderderiv},
we can derive a sharper  point-wise estimate and analyze higher-order derivatives of interpolation errors.  \qed
\end{rem}

\subsection{Analysis of quadrature errors}
 Recall the interpolatory
Gegenbauer-Gauss-type quadrature formula:
\begin{equation}\label{formulass}
\int_{-1}^1 u(x) (1-x^2)^{\lambda-1/2} dx \approx  \sum_{j=0}^n u(x_j) \omega_j=
\int_{-1}^1 (I_nu)(x) (1-x^2)^{\lambda-1/2} dx,
\end{equation}
where the quadrature weights $\{\omega_j\}_{j=0}^n$  are expressed by the Lagrange basis polynomials (see, e.g., \cite{szeg75,Funa92}).
Observe that
\begin{equation}\label{formulaerr}
\Big|\int_{-1}^1 (u-I_nu)(x) (1-x^2)^{\lambda-1/2} dx\Big|\le h_0^{\lambda} \max_{|x|\le 1}|(u-I_{n} u)(x)|,
\end{equation}
where $h_0^\lambda$ is given by \eqref{gammafd}.  With the aid of interpolation
error estimates  in Theorem \ref{theorem3.1} and Theorem
\ref{theorem3.3}, we are able to derive the quadrature errors immediately.

\subsection{Numerical results}
In what follows, we provide two numerical examples to demonstrate the sharpness of the estimates established in
Theorems \ref{Th:spectraldiff} and \ref{theorem3.3}.

\subsubsection{Example 1}  We take
 \begin{equation}\label{test1}
u(x)=\frac 1 {x^2+1},
\end{equation}
which has two simple poles at $\pm \ri.$ By \eqref{semiaxis}, we are free to  choose
 ${\mathcal E}_\rho$ with $\rho$ in the range
\begin{equation}\label{achoos}
1<\rho<1+\sqrt 2 \approx 2.414,
\end{equation}
such that $u$ is analytic on and within ${\mathcal E}_\rho.$ To
compare the (discrete) maximum error of spectral differentiation
with the upper bound, we sample about $2000$ values of $\rho$
equally from $(1, 1+\sqrt 2),$ and find a tighter upper bound (which
is usually attained when $\rho$ is close to  $1+\sqrt 2$\,). In
Figure \ref{figbnd2} (a)-(b), we plot the (discrete) maximum errors
of spectral differentiation against the upper bounds. We visualize
the exponential decay of the errors, and the upper bounds and the
errors decay at almost the same rate. Moreover, it seems that the
bounds are slightly sharper in the Gegenbauer-Gauss case.

\subsubsection{Example 2} The  estimates indicates that the errors essentially depend on the location of singularity (although it affects the constant $M_{\rho}$) rather than the behavior of the singularity. To show this, we test the function with poles at $\pm \ri$ of order $2:$
\begin{equation}\label{test2}
u(x)=\frac 1 {(x^2+1)^2}.
\end{equation}
 We plot in Figure  \ref{figbnd2} (c)-(d) the errors and upper bounds as in (a)-(b). Indeed, a similar convergence behavior is observed. Indeed,
 Boyd  \cite{Boyd01} pointed out that the type of singularity might change the rate of convergence by a power of $n,$ but not an exponential function of $n.$

\begin{figure}[!t]
\subfigure[Upper bound vs Max.-error (Example 1)]{
\begin{minipage}[t]{0.47\textwidth}
\centering
\rotatebox[origin=cc]{-0}{\includegraphics[width=2.5in]{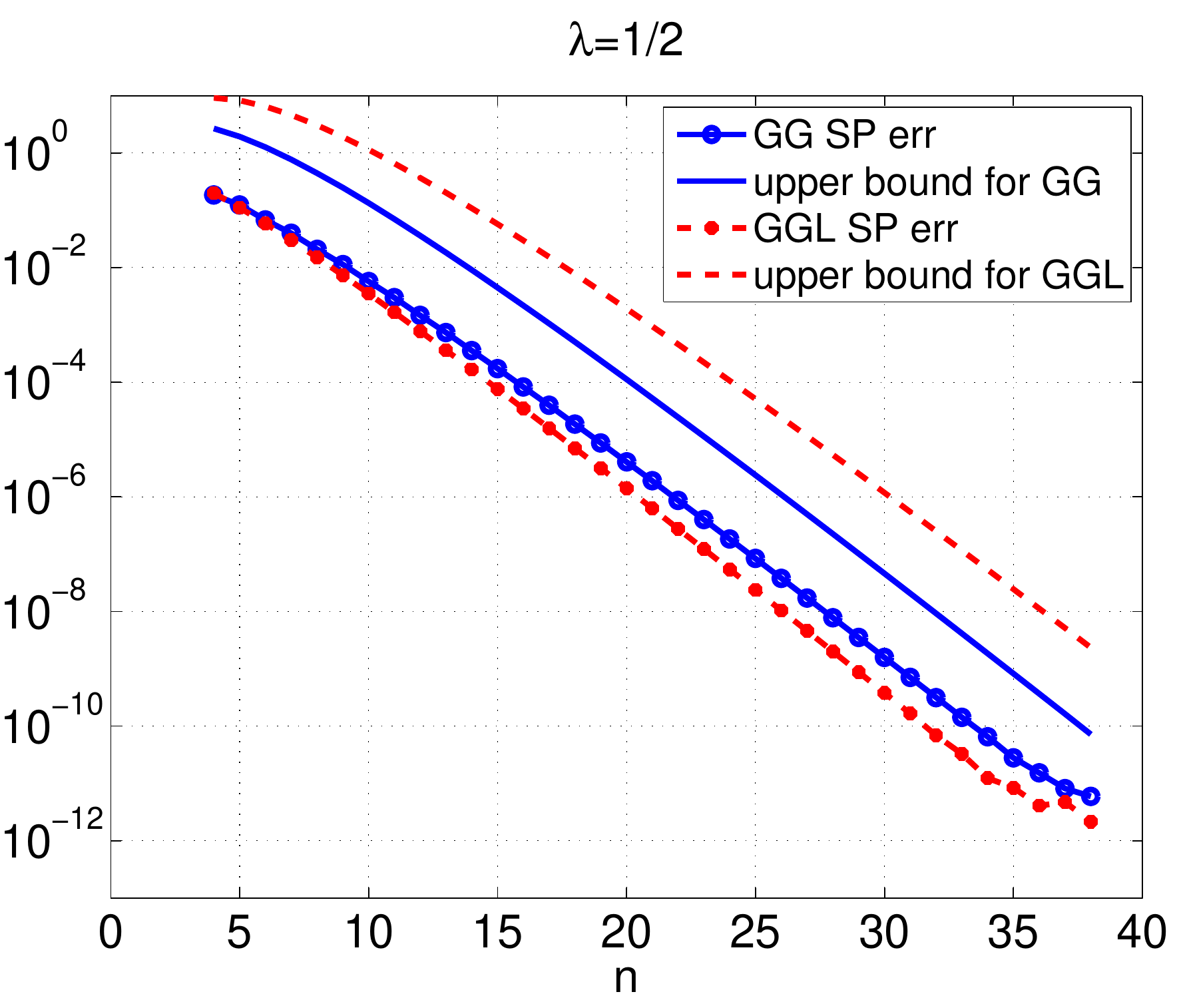}}
\end{minipage}}%
\subfigure[Upper bound vs Max.-error (Example 1)]{
\begin{minipage}[t]{0.47\textwidth}
\centering
\rotatebox[origin=cc]{-0}{\includegraphics[width=2.5in]{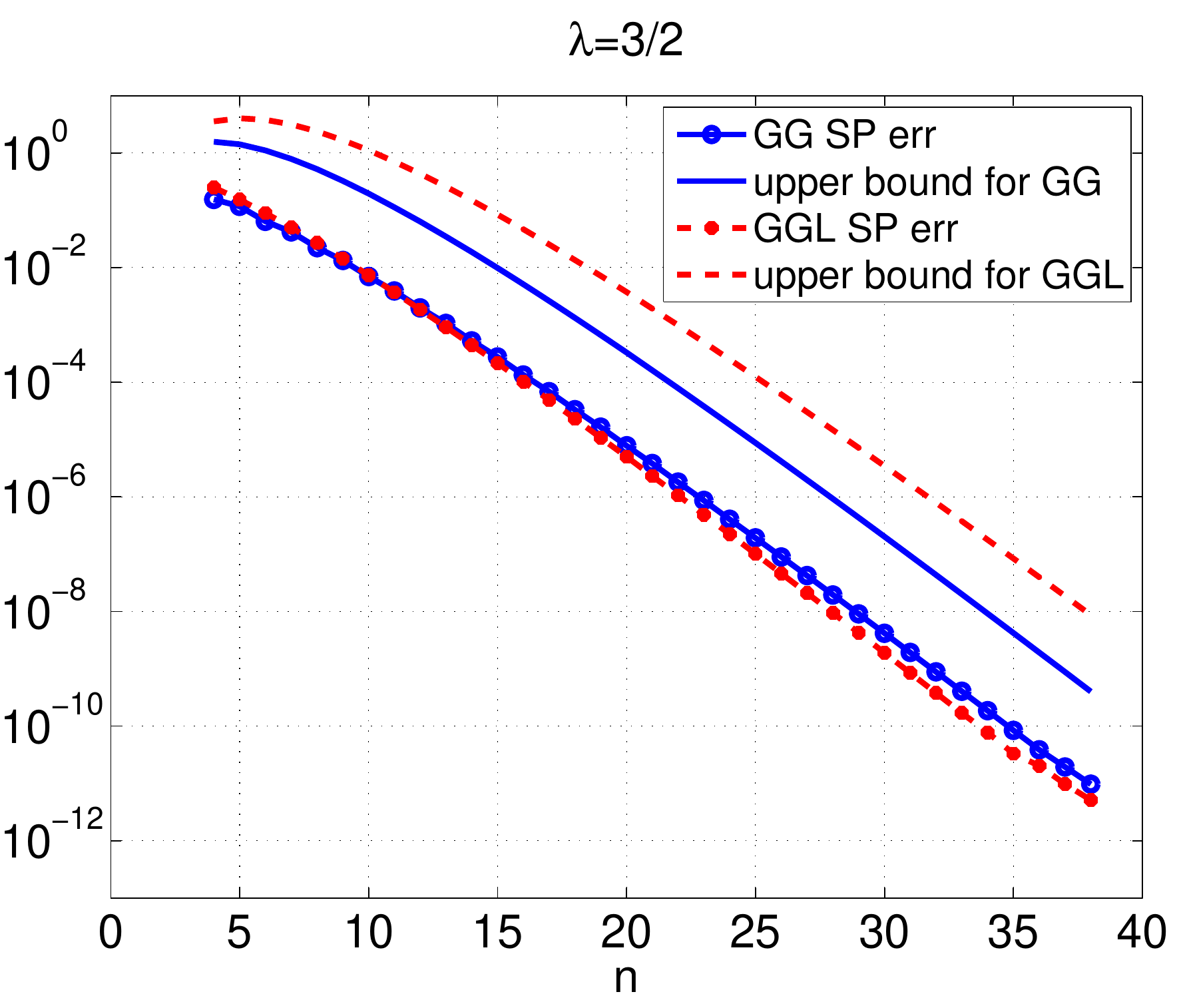}}
\end{minipage}}%

\subfigure[Upper bound vs Max.-error (Example 2)]{
\begin{minipage}[t]{0.47\textwidth}
\centering
\rotatebox[origin=cc]{-0}{\includegraphics[width=2.5in]{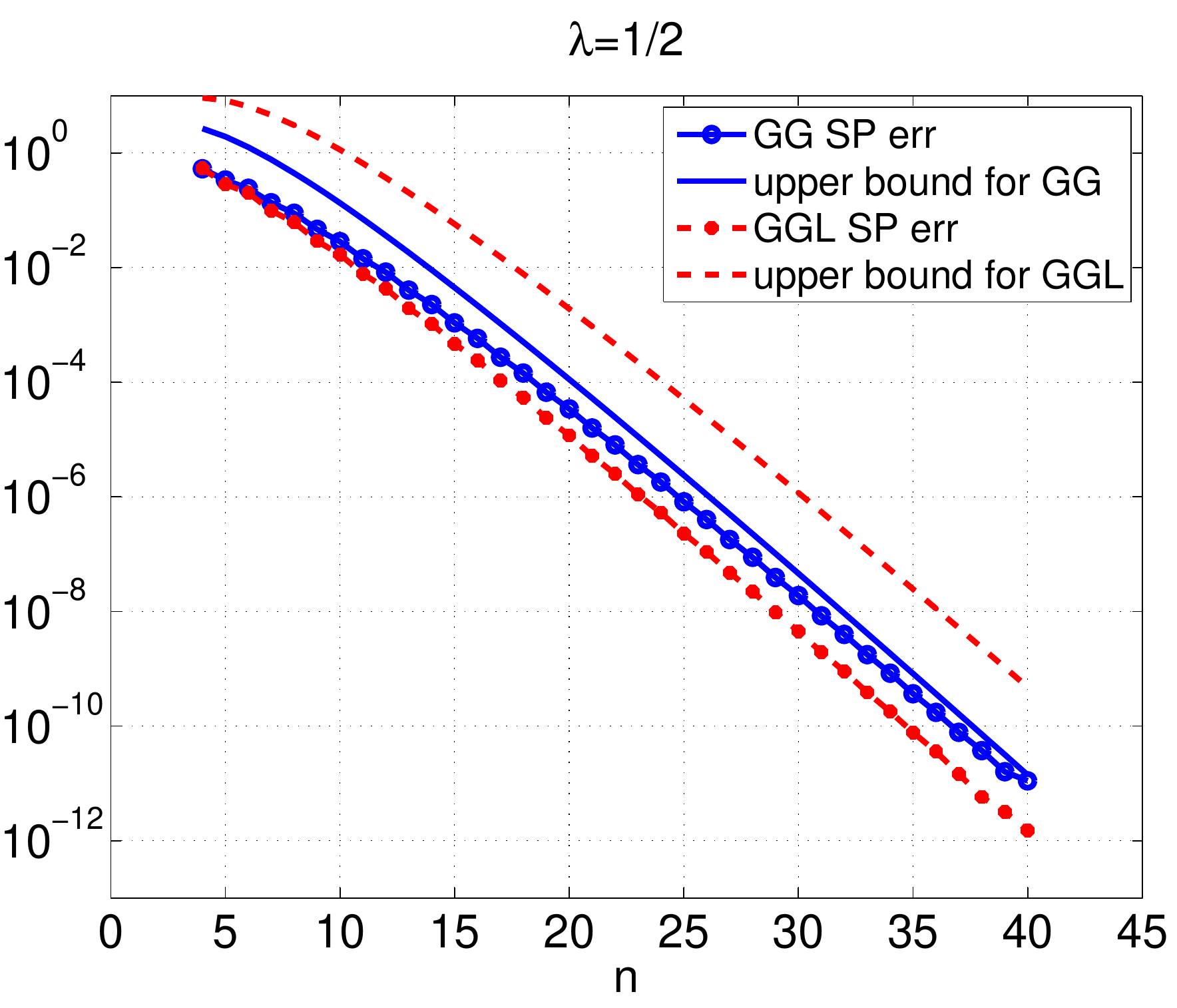}}
\end{minipage}}%
\subfigure[Upper bound vs Max.-error (Example 2)]{
\begin{minipage}[t]{0.47\textwidth}
\centering
\rotatebox[origin=cc]{-0}{\includegraphics[width=2.5in]{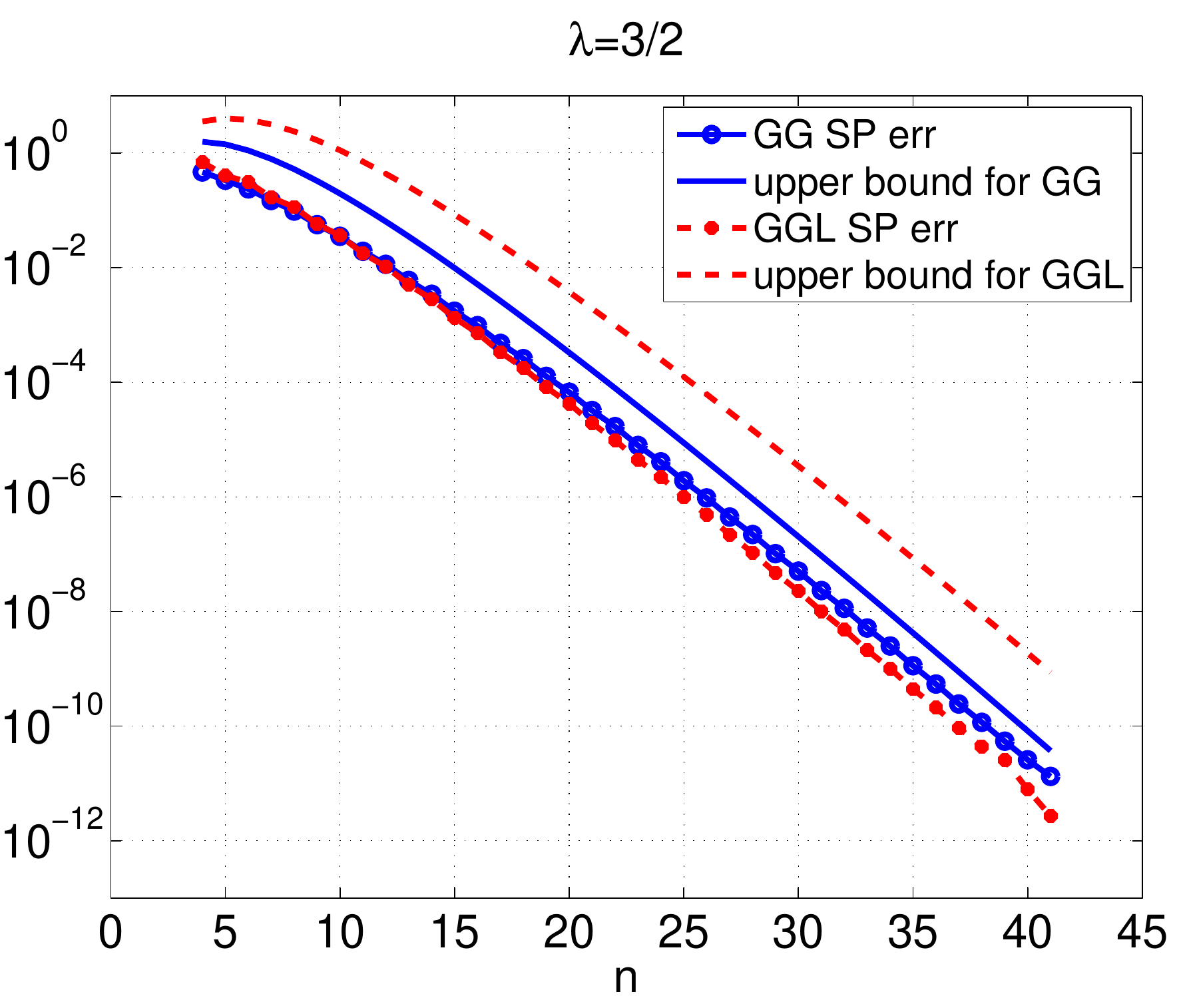}}
\end{minipage}}%

\caption{\small  The (discrete) maximum errors of Gegenbauer-Gauss and
Gegenbauer-Gauss-Lobatto spectral differentiation against the upper bounds
in  Theorems \ref{Th:spectraldiff} and \ref{theorem3.3} with $\lambda=1/2, 3/2.$
 Example 1 (a)-(b), and Example 2 (c)-(d). In the legend,
  GG SP err (resp. GGL SP err) represents the Gegenbauer-Gauss (resp.
  Gegenbauer-Gauss-Lobatto)
spectral differentiation error.} \label{figbnd2}
  \end{figure}

\vskip 20pt
 \noindent \underline{\large\bf Concluding remarks}
\vskip 10pt In this paper, the exponential convergence of Gegenbauer
interpolation, spectral differentiation and quadrature of functions
analytic on and within a sizable ellipse is analyzed. Sharp
estimates in the maximum norm with explicit dependence on the
important parameters are obtained.  Illustrative numerical results
are provided to support the analysis. For clarity of presentation,
it is assumed that  $\lambda$ is fixed in our analysis, but the
dependence of the error on  this parameter can also be tracked if it
is necessary.

\vskip 10pt
\noindent \underline{\large\bf Acknowledgement}
\vskip 5pt

The first author would like to thank the Division of Mathematical Sciences at Nanyang Technological University in Singapore for the hospitality during the visit.
The second author is grateful to Prof. Jie Shen at Purdue University for bringing this
topic to attention. The authors would also like to thank the anonymous referees for their valuable comments and suggestions.

\vskip 15pt

\begin{appendix}

\renewcommand{\theequation}{A.\arabic{equation}}

\section{Exponential convergence of Gegenbauer expansions}\label{AppendixAdd}

To this end, we discuss the exponential convergence of Gegenbauer
polynomial (with fixed  $\lambda$) expansions  of functions
satisfying the assumption of analyticity in
\cite{gottlieb1992gibbs}, that is, \eqref{Anconnet} in the form:
\begin{equation}\label{ShuGotform}
\max_{|x|\le 1}\Big|\frac{d^k u}{dx^k}(x)\Big|\le C(\delta)
\frac{k!}{\delta^k},\quad \delta\ge 1,\;\; \forall k\ge 0,
\end{equation}
where $C(\delta)$ is a positive constant depending only on $\delta.$
We write
\[
u(x)=\sum_{l=0}^\infty \hat u_l^\lambda C_l^\lambda(x)\;\;\; {\rm
with} \;\;\;  \hat u_l^\lambda =\frac 1 {h_l^\lambda} \int_{-1}^1
u(x)C_l^\lambda (x) (1-x^2)^{\lambda-1/2}dx,
\]
where $h_l^\lambda$ is defined in \eqref{gammafd}.
\begin{prop}\label{propadd1}  If $u(x)$ is an analytic function on $[-1,1]$ satisfying the assumption
\eqref{ShuGotform}, then for fixed $\lambda>-1/2,$ we have
\begin{equation}\label{modelmm}
\max_{|x|\le 1}\big|(\pi_n^\lambda u-u)(x)\big|\le
C(\delta,\lambda)\frac{n^\lambda}{(2\delta)^n},
\end{equation}
where $(\pi_n^\lambda u)(x)=\sum_{l=0}^n \hat u_l^\lambda
C_l^\lambda(x),$ and $C(\delta,\lambda)$ is a positive constant
depending  on $\delta$ and $\lambda.$
\end{prop}
 \begin{proof} The estimate follows directly from  Theorem 4.3 in \cite{gottlieb1992gibbs}.  Indeed, by Theorem 4.3 in  \cite{gottlieb1992gibbs}, we have
 \begin{equation*}
 \begin{split}
 \max_{|x|\le 1}\big|(\pi_n^\lambda u-u)(x)\big|\le A \frac{C(\delta)\Gamma(\lambda+1/2)\Gamma(n+2\lambda+1)}
 {n\sqrt \lambda (2\delta)^n \Gamma(2\lambda)\Gamma(n+\lambda)}, \quad n\ge 1,\; \lambda>-1/2,
 \end{split}
 \end{equation*}
 where  $A$ is a generic positive constant  independent of $\lambda, n$ and $u.$ For fixed $\lambda>-1/2,$ using
 \eqref{stirling} leads to
\begin{equation*}
\begin{split}
 \frac{\Gamma(n+2\lambda+1)}{\Gamma(n+\lambda+1)}&= \frac{n+\lambda+1}{n+2\lambda+1}\frac{\Gamma(n+2\lambda+2)}{\Gamma(n+\lambda+2)}\\
 &\le A (n+2\lambda+1)^\lambda
 \Big(1+\frac \lambda {n+\lambda+1}\Big)^{n+\lambda+3/2} e^{-\lambda}\\
 &\le A (n+2\lambda+1)^\lambda e^{\frac{\lambda}{2(n+\lambda+1)}},
\end{split}
 \end{equation*}
 where in the last step, we used the inequality $1+x\le e^x$ for all real $x.$ Therefore, we obtain the estimate
 \begin{equation*}
 \begin{split}
 \max_{|x|\le 1}\big|(\pi_n^\lambda u-u)(x)\big|\le C(\delta,\lambda) \frac{n+\lambda}{ n}\frac{(n+2\lambda+1)^\lambda}{(2\delta)^n}.
 \end{split}
 \end{equation*}
This implies \eqref{modelmm}.
\end{proof}

\begin{rem}\label{Augadd1} The error  $\max_{|x|\le 1}\big|(\pi_n^\lambda u-u)(x)\big|$ is termed as the regularization error in \cite{gottlieb1992gibbs,Got.S97}. This,  together with the so-called truncation error, contributes to the error of the Gegenbauer reconstruction. As shown in \cite{gottlieb1992gibbs}, an exponential convergence could be recovered from Fourier partial sum of a nonperiodic  analytic function, when $\lambda$ grows linearly with $n.$ However, the estimate \eqref{modelmm} shows that the regularization error alone decays exponentially  for fixed $\lambda$. \qed
\end{rem}

\begin{rem}\label{Augadd2}  It is also interesting to study the exponential convergence of Gegenbauer expansions of  analytic functions, which are analytic on and within the Bernstein ellipse ${\mathcal E}_\rho$ in \eqref{Berellips}.
We refer to
\cite{davis1975interpolation,rivlin1990chebyshev,MR1937591} for the
results on  Chebyshev expansions. However, the analysis for the
Legendre expansions is much more involved. Davis
\cite{davis1975interpolation}  stated an  estimate  due to K.
Neumann (see Page 312 of \cite{davis1975interpolation}), that is,
if $u$ is analytic on and inside ${\mathcal E}_\rho$ with $\rho>1,$
then the Legendre expansion coefficient satisfies
\begin{equation}\label{bestiim}
\underset{{n\to\infty}}{\lim\sup}\, |\hat u_n|^{1/n}=\rho^{-1}.
\end{equation}
A more informative estimate:
\begin{equation}\label{bestiimaa}
|\hat u_n|\le C(\rho) \frac n {\rho^{n+1}},\quad \forall n\ge 0,
\end{equation}
was presented in the very recent paper \cite{Xiang2011}. Some
discussions on the Gegenbauer expansions can be found in
\cite{Got.S97}.  We shall report the analysis for the  general
Jacobi expansions in \cite{WangXieZhao2011}.  \qed
\end{rem}

\vskip 10pt

\renewcommand{\theequation}{B.\arabic{equation}}

\section{Proof of Lemma \ref{lemma1.4}}\label{AppendixA}

 We carry out the proof by induction.

Apparently,  by \eqref{PropG3},  $C_0^{\lambda}(z)=1,$ so
\eqref{Gexp}-\eqref{ajexp0} holds for  $n=0.$ Similarly, we can
verify the case with $n=1.$

Assume that the formula holds for $C_{n-2}^{\lambda}(z)$ and
$C_{n-1}^{\lambda}(z)$ with $ n\ge 2$.  It follows from the
 three-term recurrence relation \eqref{PropG3} that
\begin{equation*}
\begin{split}
nC_n^{\lambda}(z)
&=2(n+\lambda-1)z\sum_{k=0}^{n-1}g_k^{\lambda}g_{n-1-k}^{\lambda}w^{n-2k-1}
\\&\quad
-(n+2\lambda-2)\sum_{k=0}^{n-2}g_k^{\lambda}g_{n-2-k}^{\lambda}w^{n-2k-2}\\
&=(n+\lambda-1)
\sum_{k=0}^{n-1}g_k^{\lambda}g_{n-1-k}^{\lambda}w^{n-2k}
+(n+\lambda-1)\sum_{k=0}^{n-1}g_k^{\lambda}g_{n-1-k}^{\lambda}w^{n-2k-2}\\
&\quad -(n+2\lambda-2)\sum_{k=0}^{n-2}g_k^{\lambda}g_{n-2-k}^{\lambda}w^{n-2k-2}\\
&=(n+\lambda-1)g_{n-1}^{\lambda} (w^n+w^{-n})+\sum_{k=1}^{n-1}
D_{n,k}^{\lambda}\, g_k^{\lambda} g_{n-k}^{\lambda} w^{n-2k},
\end{split}
\end{equation*}
where
\[
 D_{n,k}^{\lambda}=(n+\lambda-1)\left(\frac{g_{n-k-1}^{\lambda}}{g_{n-k}^{\lambda}}
 +\frac{g_{k-1}^{\lambda}}{g_{k}^{\lambda}}\right)-(n+2\lambda-2) \frac{g_{n-k-1}^{\lambda}}{g_{n-k}^{\lambda}}
 \frac{g_{k-1}^{\lambda}}{g_{k}^{\lambda}}.
\]
One verifies from \eqref{ajexp0} that
\begin{equation}\label{oneform}
{g_n^{\lambda}} =\frac{n+\lambda-1}{n}{g_{n-1}^{\lambda}},\quad
D_{n,k}^{\lambda}=n.
\end{equation}
This completes the induction. \qed

\end{appendix}

\bibliography{GegInterpAnal2010Dec}

\begin{thebibliography}{10}

\bibitem{Andrews99}
G.E. Andrews, R.~Askey, and R.~Roy.
\newblock {\em {Special Functions}}, volume~71 of {\em Encyclopedia of
  Mathematics and its Applications}.
\newblock Cambridge University Press, Cambridge, 1999.

\bibitem{Bab.G01}
I.~Babu{\v{s}}ka and B.Q. Guo.
\newblock Direct and inverse approximation theorems for the $p$-version of the
  finite element method in the framework of weighted {B}esov spaces. {I}.
  {A}pproximability of functions in the weighted {B}esov spaces.
\newblock {\em SIAM J. Numer. Anal.}, 39(5):1512--1538 (electronic), 2001/02.

\bibitem{BernardDaugbook1999}
C.~Bernardi, M.~Dauge, and Y.~Maday.
\newblock {\em {Spectral Methods for Axisymmetric Domains}}, volume~3 of {\em
  Series in Applied Mathematics (Paris)}.
\newblock Gauthier-Villars, \'Editions Scientifiques et M\'edicales Elsevier,
  Paris, 1999.
\newblock Numerical algorithms and tests due to Mejdi Aza{\"{\i}}ez.

\bibitem{MR1470226}
C.~Bernardi and Y.~Maday.
\newblock {Spectral Methods}.
\newblock In {\em {Handbook of Numerical Analysis, {V}ol. {V}}}, Handb. Numer.
  Anal., V, pages 209--485. North-Holland, Amsterdam, 1997.

\bibitem{Boyd94}
J.P. Boyd.
\newblock The rate of convergence of {F}ourier coefficients for entire
  functions of infinite order with application to the {W}eideman-{C}loot
  sinh-mapping for pseudospectral computations on an infinite interval.
\newblock {\em J. Comput. Phys.}, 110(2):360--372, 1994.

\bibitem{Boyd01}
J.P. Boyd.
\newblock {\em {C}hebyshev and {F}ourier {S}pectral {M}ethods}.
\newblock Dover Publications Inc., 2001.

\bibitem{boyd2005trouble}
J.P. Boyd.
\newblock {Trouble with Gegenbauer reconstruction for defeating Gibbs'
  phenomenon: Runge phenomenon in the diagonal limit of Gegenbauer polynomial
  approximations}.
\newblock {\em J. Comput. Phys.}, 204(1):253--264, 2005.

\bibitem{CHQZ06}
C.~Canuto, M.Y. Hussaini, A.~Quarteroni, and T.A. Zang.
\newblock {\em {Spectral Methods: Fundamentals in Single Domains}}.
\newblock Springer, Berlin, 2006.

\bibitem{Chen.T10}
Y.P. Chen and T.~Tang.
\newblock Convergence analysis of the {J}acobi spectral-collocation methods for
  {V}olterra integral equations with a weakly singular kernel.
\newblock {\em Math. Comp.}, 79(269):147--167, 2010.

\bibitem{davis1975interpolation}
P.J. Davis.
\newblock {\em {Interpolation and Approximation}}.
\newblock Dover Publications, Inc, New Year, 1975.

\bibitem{MR760629}
P.J. Davis and P.~Rabinowitz.
\newblock {\em Methods of Numerical Integration}.
\newblock Computer Science and Applied Mathematics. Academic Press Inc.,
  Orlando, FL, second edition, 1984.

\bibitem{Doha91}
E.H. Doha.
\newblock The coefficients of differentiated expansions and derivatives of
  ultraspherical polynomials.
\newblock {\em Comput. Math. Appl.}, 21(2-3):115--122, 1991.

\bibitem{Doha98}
E.H. Doha.
\newblock The ultraspherical coefficients of the moments of a general-order
  derivative of an infinitely differentiable function.
\newblock {\em J. Comput. Appl. Math.}, 89(1):53--72, 1998.

\bibitem{DohaBhrawy06}
E.H. Doha and A.H. Bhrawy.
\newblock Efficient spectral-{G}alerkin algorithms for direct solution for
  second-order differential equations using {J}acobi polynomials.
\newblock {\em Numer. Algorithms}, 42(2):137--164, 2006.

\bibitem{DohaBhrawy09}
E.H. Doha and A.H. Bhrawy.
\newblock A {J}acobi spectral {G}alerkin method for the integrated forms of
  fourth-order elliptic differential equations.
\newblock {\em Numer. Methods Partial Differential Equations}, 25(3):712--739,
  2009.

\bibitem{DohaBhrawyAbd09}
E.H. Doha, A.H. Bhrawy, and W.M. Abd-Elhameed.
\newblock Jacobi spectral {G}alerkin method for elliptic {N}eumann problems.
\newblock {\em Numer. Algorithms}, 50(1):67--91, 2009.

\bibitem{Dub.91}
M.~Dubiner.
\newblock Spectral methods on triangles and other domains.
\newblock {\em J. Sci. Comput.}, 6(4):345--390, 1991.

\bibitem{Elliott71}
David Elliott.
\newblock Uniform asymptotic expansions of the {J}acobi polynomials and an
  associated function.
\newblock {\em Math. Comp.}, 25:309--315, 1971.

\bibitem{fornberg1998practical}
B.~Fornberg.
\newblock {\em {A Practical Guide to Pseudospectral Methods}}.
\newblock Cambridge Univ Pr, 1998.

\bibitem{Funa92}
D.~Funaro.
\newblock {\em Polynomial Approxiamtions of Differential Equations}.
\newblock Springer-Verlag, 1992.

\bibitem{gelb2006determining}
A.~Gelb and Z.~Jackiewicz.
\newblock {Determining analyticity for parameter optimization of the Gegenbauer
  reconstruction method}.
\newblock {\em SIAM J. Sci. Comput.}, 27(3):1014--1031, 2006.

\bibitem{gottlieb1977numerical}
D.~Gottlieb and S.A. Orszag.
\newblock {\em {Numerical Analysis of Spectral Methods: Theory and
  Applications}}.
\newblock Society for Industrial Mathematics, 1977.

\bibitem{Got.S97}
D.~Gottlieb and C.W. Shu.
\newblock On the {G}ibbs phenomenon and its resolution.
\newblock {\em SIAM Rev.}, 39(4):644--668, 1997.

\bibitem{gottlieb1992gibbs}
D.~Gottlieb, C.W. Shu, A.~Solomonoff, and H.~Vandeven.
\newblock On the {G}ibbs phenomenon. {I}. {R}ecovering exponential accuracy
  from the {F}ourier partial sum of a nonperiodic analytic function.
\newblock {\em J. Comput. Appl. Math.}, 43(1-2):81--98, 1992.
\newblock Orthogonal polynomials and numerical methods.

\bibitem{GuoBQ2011}
B.Q. Guo.
\newblock {\em {Mathematical Foundation and Applications of the $p$ and $h$-$p$
  Finite Element Methods}}, volume~4 of {\em Series in Analysis}.
\newblock World Scientific, 2011.

\bibitem{GuoBy2001}
B.Y. Guo.
\newblock Gegenbauer approximation and its applications to differential
  equations with rough asymptotic behaviors at infinity.
\newblock {\em Appl. Numer. Math.}, 38(4):403--425, 2001.

\bibitem{Guo.W00}
B.Y. Guo and L.L. Wang.
\newblock {Jacobi interpolation approximations and their applications to
  singular differential equations}.
\newblock {\em Adv. Comput. Math.}, 14(3):227--276, 2001.

\bibitem{Guo.W04}
B.Y. Guo and L.L. Wang.
\newblock {Jacobi approximations in non-uniformly Jacobi-weighted Sobolev
  spaces}.
\newblock {\em J. Approx. Theory}, 128(1):1--41, 2004.

\bibitem{gustafsson2011work}
B.~Gustafsson.
\newblock {The work of David Gottlieb: A success story}.
\newblock {\em Commun. Comput. Phys.}, 9(3):481--496, 2011.

\bibitem{Jeffrey:1086465}
A.~Jeffrey.
\newblock {\em Table of Integrals, Series, and Products (Seventh Edition)}.
\newblock Elsevier, San Diego, CA, 2007.

\bibitem{Kar.S05}
G.E. Karniadakis and S.J. Sherwin.
\newblock {\em Spectral/{$hp$} Element Methods for Computational Fluid
  Dynamics}.
\newblock Numerical Mathematics and Scientific Computation. Oxford University
  Press, New York, second edition, 2005.

\bibitem{krasikov2007upper}
I.~Krasikov.
\newblock {An upper bound on Jacobi polynomials}.
\newblock {\em J. Approx. Theory}, 149(2):116--130, 2007.

\bibitem{MR1937591}
J.C. Mason and D.C. Handscomb.
\newblock {\em Chebyshev Polynomials}.
\newblock Chapman \& Hall/CRC, Boca Raton, FL, 2003.

\bibitem{NevaiPEr1994}
P.~Nevai, T.~Erd{\'e}lyi, and A.P. Magnus.
\newblock Generalized {J}acobi weights, {C}hristoffel functions, and {J}acobi
  polynomials.
\newblock {\em SIAM J. Math. Anal.}, 25(2):602--614, 1994.

\bibitem{ReddyWeideman2005}
S.C. Reddy and J.A.C. Weideman.
\newblock The accuracy of the {C}hebyshev differencing method for analytic
  functions.
\newblock {\em SIAM J. Numer. Anal.}, 42(5):2176--2187 (electronic), 2005.

\bibitem{rivlin1990chebyshev}
T.J. Rivlin and V.~Kalashnikov.
\newblock {\em {Chebyshev Polynomials: From Approximation Theory to Algebra and
  Number Theory}}.
\newblock Wiley New York, 1990.

\bibitem{ShenTangWang2011}
J.~Shen, T.~Tang, and L.L. Wang.
\newblock {\em {Spectral Methods : Algorithms, Analysis and Applications}},
  volume~41 of {\em Series in Computational Mathematics}.
\newblock Springer, 2011.

\bibitem{szeg75}
G.~Szeg\"o.
\newblock {\em Orthogonal Polynomials (Fourth Edition)}.
\newblock AMS Coll. Publ., 1975.

\bibitem{tadmor1986exponential}
E.~Tadmor.
\newblock {The exponential accuracy of Fourier and Chebyshev differencing
  methods}.
\newblock {\em SIAM J. Numer. Anal.}, 23(1):1--10, 1986.

\bibitem{Tref00}
L.N. Trefethen.
\newblock {\em Spectral Methods in {Matlab}}.
\newblock Software, Environments, and Tools. Society for Industrial and Applied
  Mathematics (SIAM), Philadelphia, PA, 2000.

\bibitem{TrefSIAMRev08}
L.N. Trefethen.
\newblock Is {G}auss quadrature better than {C}lenshaw-{C}urtis?
\newblock {\em SIAM Rev.}, 50(1):67--87, 2008.

\bibitem{Xiang2011}
H.Y. Wang and S.H. Xiang.
\newblock {On the convergence rates of Legendre approximation}.
\newblock {\em Math. Comput.}, In press, 2011.

\bibitem{WangXieZhao2011}
L.L. Wang, Z.Q. Xie, and X.D. Zhao.
\newblock { A note on exponential convergence of Jacobi polynomial expansions}.
\newblock In preparation, 2011.

\bibitem{WangGuo08}
Z.Q. Wang and B.Y. Guo.
\newblock Jacobi rational approximation and spectral method for differential
  equations of degenerate type.
\newblock {\em Math. Comp.}, 77(262):883--907, 2008.

\bibitem{watson66}
G.N. Watson.
\newblock {\em A Treatise on the Theory of Bessel Functions}.
\newblock Cambridge University Press, 1995.

\bibitem{ZhangZM04}
Z.~Zhang.
\newblock Superconvergence of spectral collocation and {$p$}-version methods in
  one dimensional problems.
\newblock {\em Math. Comp.}, 74(252):1621--1636 (electronic), 2005.

\bibitem{ZhangZM08}
Z.~Zhang.
\newblock Superconvergence of a {C}hebyshev spectral collocation method.
\newblock {\em J. Sci. Comput.}, 34(3):237--246, 2008.

\end{thebibliography}

\end{document}